\pdfoutput=1
\documentclass[12pt]{article}

\usepackage[english]{babel}

\newcommand{\suppress}[1]{}

\usepackage[T1]{fontenc}
\usepackage{booktabs}
\usepackage{amsmath,amsthm,amssymb}
\usepackage{mathtools}
\usepackage{physics}
\usepackage{algorithm}
\usepackage{algorithmic}
\usepackage{url}
\usepackage{enumerate}
\usepackage[colorlinks=true, allcolors=blue]{hyperref}
\usepackage{xspace}
\usepackage{upgreek}
\usepackage{microtype}

\hypersetup{citecolor=black, linkcolor=black, colorlinks=true}
\numberwithin{equation}{section}
\usepackage{cleveref}
\usepackage{thmtools}

\newtheorem{theorem}{Theorem}
\newtheorem{corollary}[theorem]{Corollary}
\newtheorem{proposition}[theorem]{Proposition}

\newtheorem{lemma}[theorem]{Lemma}

\newtheorem{remark}[theorem]{Remark}

\newcommand{\qi}{\mathbf{i}}
\newcommand{\qj}{\mathbf{j}}
\newcommand{\qk}{\mathbf{k}}

\newcommand{\C}{\mathbb{C}}
\newcommand{\N}{\mathbb{N}}
\newcommand{\HH}{\mathbb{H}}

\newcommand{\Q}{\mathcal{Q}}

\newcommand{\R}{\mathbb{R}}

\newcommand*{\email}[1]{\texttt{#1}}

\newcommand{\complex}{{\mathbb C}}
\newcommand{\reals}{{\mathbb R}}

\newcommand{\union}{\cup}

\newcommand{\adjoint}{*}
\newcommand{\eqdef}{\coloneqq}
\newcommand{\eq}{=}
\newcommand{\diracdelta}{\updelta}

\newcommand{\set}[1]{\left\{ #1 \right\}}

\DeclareMathOperator{\e}{e}
\newcommand{\complexi}{{\mathrm{i}}}
\newcommand{\id}{{\mathrm I}}

\newcommand{\qsphere}{{\mathsf S}}
\newcommand{\so}{\mathsf{SO}}

\DeclareMathOperator{\Span}{span}
\DeclareMathOperator{\diag}{diag}

\DeclareMathOperator{\PSD}{PSD}
\DeclareMathOperator{\CPSD}{CPSD}
\DeclareMathOperator{\DFT}{DFT}

\newcommand{\tH}{\widetilde{H}}

\newcommand{\cL}{{\mathcal L}}
\newcommand{\cS}{{\mathcal S}}

\newcommand{\detii}{\genfrac{\lvert}{\rvert}{0pt}{}}

\newcommand{\PauliX}{\upsigma_{\mathrm{x}}}
\newcommand{\PauliY}{\upsigma_{\mathrm{y}}}
\newcommand{\PauliZ}{\upsigma_{\mathrm{z}}}

\newcommand{\com}{/\hspace{-1pt}/ }

\makeatletter
\newcommand{\subjclass}[2][1991]{%
  \let\@oldtitle\@title%
  \gdef\@title{\@oldtitle\footnotetext{#1 \emph{Mathematics subject classification.} #2}}%
}
\newcommand{\keywords}[1]{%
  \let\@@oldtitle\@title%
  \gdef\@title{\@@oldtitle\footnotetext{\emph{Key words and phrases.} #1.}}%
}
\makeatother

\title{Quaternionic Perfect Sequences \\ and Hadamard Matrices}
\date{January 29, 2026}
\author{Aidan Bennett\thanks{School of Computer Science, University of Windsor.
401 Sunset Ave.\ Windsor,
Canada.
Email: \email{bennet43@uwindsor.ca}.
Research supported by an NSERC undergraduate research assistantship.}
\and
Curtis Bright\thanks{School of Computer Science, University of Waterloo.
200 University Ave.\ W., Waterloo,
Canada.
Email: \email{cbright@uwaterloo.ca}, Webpage: \url{www.curtisbright.com}.
Research supported in part by an NSERC Discovery Grant.
Corresponding author.}
\and
Paul Colinot\thanks{Université Grenoble Alpes, G-SCOP, CNRS\@.
681 Rue de la Passerelle,
Saint-Martin-d'Hères, France.
Email: \email{paul.colinot@grenoble-inp.org}.
Research conducted under the Mitacs Globalink research internship program.}
\and 
Ashwin Nayak\thanks{Department of Combinatorics and Optimization
and Institute for Quantum Computing, University of Waterloo.
200 University Ave.\ W., Waterloo,
Canada.
Email: \email{academic@ashwinnayak.info}. 
Research supported in part by an NSERC Discovery Grant.}
}

\begin{document}
\maketitle

\begin{abstract}
A finite sequence of numbers is perfect if it has zero periodic autocorrelation after a nontrivial cyclic shift.
In this work, we study quaternionic perfect sequences having a one-to-one correspondence
with the binary sequences arising in Williamson's construction of quaternion-type Hadamard matrices.
Using this correspondence, we devise an enumeration algorithm
that is significantly faster than previously used algorithms and does not require the sequences to be symmetric.
We implement our algorithm and use it to enumerate all circulant and possibly non-symmetric Williamson-type matrices of orders up to 21;
previously, the largest order exhaustively enumerated was 13.
We prove that when the blocks of a quaternion-type Hadamard matrix are circulant,
the blocks are necessarily pairwise amicable. This dramatically improves the filtering power of our algorithm:
in order~20, the number of block pairs needing consideration is reduced by a factor of over 25,000.
We use our results to construct quaternionic Hadamard matrices of interest in quantum communication and prove they are not equivalent to those constructed by other means.
We also study the properties of quaternionic Hadamard matrices analytically,
and demonstrate the feasibility of characterizing quaternionic Hadamard matrices with a fixed pattern of entries.
These results indicate a richer set of properties and suggest an abundance of quaternionic Hadamard matrices for sufficiently large orders.
\end{abstract}

\section{Introduction}\label{sec:intro}

A finite sequence is called \emph{perfect} if it has no correlation with itself after a nontrivial cyclic shift.
A matrix $H\in\{\pm1\}^{n\times n}$ is a (real) \emph{Hadamard matrix} if it satisfies $HH^\top=n\id_n$, and a matrix is \emph{circulant}
if every row is a cyclic shift of the previous row by one position to the right.  
When its entries are over an alphabet of $\pm1$, a perfect sequence is equivalent to a Hadamard matrix that is circulant~\cite{schmidt}.
For example,
\[ \begin{bmatrix}
1 & -1 & -1 & -1 \\
-1 & 1 & -1 & -1 \\
-1 & -1 & 1 & -1 \\
-1 & -1 & -1 & 1
\end{bmatrix} \]
is a circulant Hadamard matrix equivalent to the perfect sequence defining its first row $[1,-1,-1,-1]$.

Perfect sequences have long been studied for both their practical applications and their elegant mathematical properties.
However, the condition of having \emph{no} correlation with itself at all nontrivial cyclic shifts is in practice
difficult to achieve for most alphabets.  In fact, over the binary $\{\pm1\}$ alphabet the longest known perfect sequence
is the length~4 sequence given above, and
it has been conjectured that this is the \emph{only} nontrivial
perfect binary sequence up to isomorphism~\cite{ryser1963combinatorial}.

By increasing the alphabet size one can find arbitrarily long perfect sequences.  One of the first constructions
for perfect sequences was in 1961 by R.~C.~Heimiller~\cite{Heimiller1961}.  He
gave a construction for perfect sequences of length $p^2$
over an alphabet consisting of the $p$th roots of unity when $p$ is prime. In fact, R.~L.~Frank
had discovered the same construction, without the restriction that the length be prime,
in the 1950s while working as an aircraft engineer at the Sperry
Gyroscope Company~\cite{frank}.
For example, Frank's construction produces perfect sequences
of length $4^2=16$ over the quaternary (complex) alphabet $\{\pm1,\pm\qi\}$.

Perfect sequences over $\{\pm1,\pm\qi\}$ are
equivalent to circulant complex Hadamard matrices, where
a matrix $H\in \complex^{n\times n}$ is a \emph{complex Hadamard matrix}
if all its entries have modulus one and~$HH^*=n\id_n$ where $H^*$ denotes the conjugate transpose of $H$.
For example, Frank's construction in length~4 generates
a circulant $16\times16$ complex Hadamard matrix
and it has been conjectured that this is the largest circulant complex Hadamard matrix~\cite{turyn1970complex}.

Although it is unlikely that there are arbitrarily large circulant Hadamard matrices,
there do exist arbitrarily large non-circulant Hadamard matrices.  For example,
in 1867, Sylvester~\cite{Sylvester1867} constructed Hadamard matrices in orders~$2^k$ for all $k\geq0$.
In 1893, Hadamard~\cite{hadamard1893resolution} proved
that if a Hadamard matrix exists of order $n>2$ then $n$ must be divisible by 4.
The \emph{Hadamard conjecture} is that this necessary condition is also sufficient.
Computational evidence supports the conjecture that Hadamard matrices do exist in all orders that are multiples of four.
Although this conjecture has not been proven, many infinite classes of Hadamard matrices have been discovered.

\subsection{Quaternion-type Hadamard matrices}
\label{sec:qt-intro}

An important method of constructing Hadamard matrices was introduced by Williamson~\cite{Williamson1944} in 1944.
He considered Hadamard matrices of the form
\[ \begin{bmatrix}
A & B & C & D \\
-B & A & -D & C \\
-C & D & A & -B \\
-D & -C & B & A
\end{bmatrix} . \]
Hadamard matrices of this form are known as \emph{quaternion-type}~\cite{Baumert1965}, since the structure is based on the matrix representation
of the fundamental quaternion units $1$, $\qi$, $\qj$, and $\qk$~\cite{hamilton}.
In Williamson's original paper, the block matrices $A$, $B$, $C$, and $D$ were pairwise commuting symmetric
$\{\pm1\}$-matrices of order~$n$ satisfying $A^2 + B^2 + C^2 + D^2 = 4n \id_n$.
Baumert and Hall~\cite{Baumert1965} point out that Williamson's pairwise commuting condition can be replaced by the
more general condition
\[ UV - VU + XY - YX = 0 \]
for $(U,V,X,Y)\in\{(A,B,C,D),(A,C,B,D),(A,D,B,C)\}$.
Additionally, the block matrices of a quaternion-type Hadamard matrix do not necessarily need to be symmetric,
but if the block matrices are not symmetric then the equations defining a quaternion-type Hadamard matrix
are given by
\begin{equation} AA^\top + BB^\top + CC^\top + DD^\top = 4n\id_n , \label{eq:sq} \end{equation}
and
\begin{equation}
\begin{split}
BA^\top  - AB^\top  + DC^\top  - CD^\top  &= 0, \\
CA^\top  - AC^\top  + BD^\top  - DB^\top  &= 0, \\
DA^\top  - AD^\top  + CB^\top  - BC^\top  &= 0.
\end{split}\label{eq:pcf}
\end{equation}
The blocks of Hadamard matrices satisfying~\eqref{eq:sq} and~\eqref{eq:pcf} are sometimes
called Williamson-type matrices, though we follow the more common convention of reserving
the name \emph{Williamson-type} for a somewhat more specific class of matrices,
namely, those satisfying~\eqref{eq:sq}
but having~\eqref{eq:pcf} replaced by the stronger \emph{pairwise amicability} condition
of $XY^\top = YX^\top$ for all $X$, $Y \in \{A,B,C,D\}$~\cite{SeberryWallis1975}.

In this paper, like in Williamson's original paper, we also require that the block matrices
$A$, $B$, $C$, $D$ be circulant, in which case we can describe the matrices in terms of their initial row.
In the case that the block matrices are circulant, we prove that
the seemingly stronger condition of pairwise amicability is actually
equivalent to the conditions~\eqref{eq:pcf} (see \Cref{thm:qt-wt-equiv}).
That is, considering the first rows of circulant
Williamson-type matrices as \emph{Williamson-type sequences} and the first rows
of circulant blocks of quaternion-type matrices as \emph{QT sequences},
we prove the notions of QT sequences and Williamson-type sequences
correspond exactly.
We provide background for understanding our results and formally define QT sequences and Williamson-type sequences
in terms of correlation in Section~\ref{sec:prelims}, with
the equivalence of QT sequences and Williamson-type sequences proven in
\Cref{sec:qt-wt-equiv}.

Prior work on circulant Williamson-type matrices also includes finding analytical constructions or obstructions. Using an analytical approach
Fitzpatrick and O'Keeffe~\cite{Fitzpatrick2023} proved that there are no Williamson-type sequences of lengths 35, 47, 53, and~59,
though their result relies on computer-assisted proofs
that there are no circulant symmetric Williamson-type matrices of order~35
(as first shown by \DJ okovi\'c~\cite{Dokovi1993}) and no circulant symmetric Williamson-type matrices of orders 47, 53, and~59
(shown by Holzmann, Kharaghani, and Tayfeh-Rezaie~\cite{Holzmann2008}).
Fitzpatrick and O'Keeffe use the amicability property of
Williamson-type matrices in their proof, so
quaternion-type Hadamard matrices with circulant blocks of orders 35, 47, 53, and~59
could conceivably exist.  The equivalence we prove in \Cref{thm:qt-wt-equiv} establishes that in fact such matrices do not.

\subsection{Quaternionic perfect sequences}

The work in this paper focuses on quaternion-type Hadamard matrices with circulant blocks, but such objects
can be shown to be equivalent to perfect sequences over a quaternionic alphabet.
The quaternions, first described by Sir William Rowan Hamilton~\cite{Hamilton1844} are an extension of the complex plane, adding two additional numbers $\qj$ and $\qk$ with $\qi^2 = \qj^2 =\qk^2 = -1$ and $\qi\qj = \qk$.
Hamilton discovered these numbers while attempting to discover ways to multiply points in three-dimensional space
in an analogous way to how complex numbers (points in two-dimensional space) can be multiplied.
Note that quaternions are not commutative in general, as $\qj\qi = -\qk$.

Kuznetsov~\cite{Kuznetsov2009} first studied perfect sequences of quaternions. He proved, among other results, that symmetries of the quaternion space preserve perfection of a sequence, and listed operations on sequences that preserve perfection.
He was able to construct a perfect quaternion sequence of length $5{,}354{,}228{,}880$, and extended many results of complex sequences to the quaternions.
In a significant development, Barrera Acevedo and Hall~\cite{AcevedoHall2012} constructed perfect sequences over the alphabet $\{\pm1,\pm\qi,\qj\}$ for all lengths $n=p^k+1\equiv2\pmod{4}$, where $p$ is prime and $k\in\N$.
For the first time this showed that perfect sequences over the basic quaternion
alphabet $\Q_8 \eqdef \{\pm1, \pm\qi, \pm\qj, \pm\qk\}$ could be found for unbounded lengths,
despite the fact that only finitely many perfect sequences are expected over the alphabet $\{\pm1,\pm\qi\}$.
Blake~\cite{BT17} ran an extensive search for perfect sequences over the alphabet $\Q_8$,
and Bright, Kotsireas, and Ganesh~\cite{BKG20-perfect-sequences} showed that perfect sequences of length~$2^k$ over $\Q_8$ exist for all~$k\in\N$.
Barrera Acevedo and Dietrich~\cite{BAD18} studied perfect quaternion sequences over $\Q_+ \eqdef \Q_8 \union q\Q_8$ where $q \eqdef (1+\qi+\qj+\qk)/2$, and built a one-to-one correspondence between perfect sequences over $\Q_+$ and quadruples of QT sequences.

In this paper we describe the results of exhaustive searches
for QT sequences (and equivalently
perfect sequences over $\Q_+$).
The previous longest length exhaustively searched was $n=13$, by Barrera Acevedo, Catháin, and Dietrich~\cite{BarreraAcevedo2019},
using an amount of compute time estimated to be up to a week~\cite{personal}.
Our algorithm exhaustively completes lengths $n\leq13$ in under~5 seconds.
Barrera Acevedo, Catháin, and Dietrich were interested in searching for
various kinds of sequences of prime length generating cocyclic Hadamard matrices (including
QT sequences as a special case),
but their search algorithm was unable to reach lengths~17 and~19.
We are able to exhaustively enumerate QT sequences
for lengths $n\leq20$ in under~8 hours and for length $n=21$ in under~40 hours
(see Section~\ref{sec:search} for our results).
The equivalence between QT sequences and Williamson-type sequences mentioned in \Cref{sec:qt-intro} plays a key role in the search, dramatically improving the amount
of filtering possible in the enumeration algorithm described in \Cref{sec:search}.
Applying this equivalence
improves the amount of filtering performed by our algorithm
by a factor of about 25,000 when enumerating QT sequences of length $n=20$.
Thus, the equivalence is not only of theoretical interest, but also of practical use.

Williamson-type matrices that are not necessarily symmetric have been studied
for over fifty years~\cite{Wallis1973,Wallis1974}, but in stark contrast to the symmetric case
we were surprised to find no
enumerations specifically of circulant Williamson-type matrices \emph{without} the assumption of symmetry.
This may be due to the fact that the assumption of symmetry dramatically decreases the search space
and simplifies the search, allowing more constraining properties to be exploited and more extensive
enumerations to be performed (including all orders~$n$ up to 60~\cite{BKG20-williamson-conjecture,Holzmann2008}).
For example, the conditions~\eqref{eq:pcf} become trivially satisfied
when the blocks $A$, $B$, $C$, $D$ are symmetric and circulant.
The enumeration
of Barrera Acevedo, Catháin, and Dietrich~\cite{BarreraAcevedo2019} of cocyclic Hadamard matrices
of order~$4p$ for primes $p\leq13$ is the most extensive exhaustive enumeration
we are aware of that does not rely on the assumption of symmetry.
Since circulant Williamson-type matrices generate
cocyclic Hadamard matrices, as a special case of their results they generated
all Williamson-type sequences of lengths $p\leq 13$ for prime $p$. The efficiencies built into the computational search in \Cref{sec:search} allow us to enumerate all Williamson-type sequences of length~$n$ for $n\leq21$.

\subsection{Quaternionic Hadamard matrices}
\label{sec-qhm-context}

A \emph{quaternionic Hadamard matrix} (QHM) of order~$n$ is an~$n \times n$ matrix~$G$ with unit quaternion entries such that~$GG^* = n\id$, where~$G^*$ is the conjugate transpose of~$G$.
Farkas, Kaniewski, and Nayak~\cite{FKN23-mum-superdense-coding} recently drew a connection between quaternionic Hadamard matrices and quantum information, rekindling interest in the construction of QHMs. In more detail, using an algebra isomorphism between quaternions and a subalgebra of the set of complex~$2 \times 2$ matrices, they used QHMs to construct Hadamard matrices of \emph{unitary operators\/}. The latter correspond to mutually unbiased measurements (MUMs), which are extensions of mutually unbiased bases (MUBs).
Consequently, every order~$n$ quaternionic Hadamard matrix corresponds to an ``$n$-outcome'' MUM\@.
We refer the interested reader to~\cite{FKN23-mum-superdense-coding} for the definitions of these concepts and their relevance in quantum information and communication.

Some MUMs can be derived from MUBs via a direct-sum construction.
Farkas, Kaniewski, and Nayak~\cite{FKN23-mum-superdense-coding} provided a complete proof of the characterization of MUMs that are direct sums of MUBs.
This characterization was
partially proven earlier by Tavakoli, Farkas, Rosset, Bancal, and Kaniewski~\cite{TFR+21-mubs}.
The characterization implies that an MUM corresponding to a normalized quaternionic Hadamard matrix
is a direct sum of MUBs if and only if all the elements of the quaternionic Hadamard matrix commute.

Farkas, Kaniewski, and Nayak showed that MUMs that are \emph{not\/} direct sums of MUBs lead to novel quantum communication protocols for superdense coding. This connection led them to the construction of new such MUMs. They presented explicit such MUMs for several orders and showed how these may be used to derive such MUMs for infinitely many orders. The MUMs were either discovered earlier by an ad hoc computational search~\cite{TFR+21-mubs}, or were derived more systematically from QHMs for a small set of orders~\cite{FKN23-mum-superdense-coding}. This motivates the study of QHMs of arbitrary order in more depth.

Chterental and {\DJ}okovi{\'c}~\cite{CD08-stochastic-matrices} observed that QHMs of order up to three are all equivalent to complex Hadamard matrices. They characterized all QHMs of order four by showing that there are precisely two families of such matrices, each of uncountably infinite cardinality. They derived explicit expressions for the matrices in terms of a few parameters. Higginbotham and Worley~\cite{HW22-qhm} continued in similar vein, studying order five QHMs of specific types and also more general constructions of QHMs.

Barrera Acevedo, Dietrich, and Lionis~\cite{BarreraAcevedo2024} provided a number of constructions of quaternionic Hadamard matrices by generalizing methods for real and complex Hadamard matrices. These constructions lead to new normalized QHMs with noncommuting entries for infinitely many orders. The constructions also include infinite families of nonequivalent such matrices for a fixed order.

Unlike the real case, quaternionic Hadamard matrices of all orders exist. The complex Fourier matrices are well-known examples.
The above works suggest the existence of multiple equivalence classes of QHMs of all orders greater than three, and even multiple infinite equivalence classes of normalized QHMs with noncommuting entries (and the corresponding MUMs) of a fixed order. Spurred by this, we continue the computational and analytical investigation of QHMs.

Circulant QHMs are equivalent to perfect quaternion sequences. The computational enumeration of QT sequences thus yields new perfect quaternionic sequences, and hence new quaternionic Hadamard matrices. The normalized matrices typically contain non-commuting entries, and hence lead to new MUMs of interest in quantum information. We list several examples of such QHMs in \Cref{subsection:patterns}.
This is in contrast to earlier computational search for non-trivial MUMs, which resulted in MUMs (those reported in~\cite{TFR+21-mubs})
for only a few small orders~\cite{Farkas21}.

We show that unlike in the case of real or complex Hadamard matrices, QHMs admit equivalence classes with uncountably many normalized matrices whenever they have at least one non-real entry in normalized form (\Cref{prop-infinite-family}).
We also propose an approach for characterizing QHMs of small order and illustrate it for order five matrices (\Cref{sec-characterization}).
Finally, we prove the order five QHMs generated by perfect $\Q_+$-sequences produce QHMs not equivalent to those
considered by Higginbotham and Worley~\cite{Higginbotham2021}, and similarly in order seven there are QHMs generated
by perfect $\Q_+$-sequences not equivalent to the QHMs generated by the perfect sequences
listed by Kuznetsov~\cite{Kuznetsov2010} and Barrera Acevedo, Dietrich, and Lionis~\cite{BarreraAcevedo2024} (see \Cref{sec-new-qhm}).
Consequently, there are at least three nonequivalent families of QHMs in orders five and seven;
along with prior work, this leads us to conjecture
there exist multiple nonequivalent quaternionic Hadamard matrices for all orders~$\ge 4$.
Our results imply this conjecture is not true if only \emph{circulant} QHMs over $\Q_+$ are considered, though.
The equivalence provided in \Cref{thm:qt-wt-equiv}, in conjunction with the results of Fitzpatrick
and O'Keeffe~\cite{Fitzpatrick2023}, imply circulant
quaternionic Hadamard matrices of orders
35, 47, 53, and~59 having entries in $\Q_+$ do not exist.

\section{Background}
\label{sec:prelims}

In this section we provide the notations and the necessary background that we use.
First, we introduce the quaternion algebra and the notation we use to denote important sets of quaternions in \Cref{sec:quat}.
We formally define Hadamard and quaternionic Hadamard matrices
and the notion of equivalence between Hadamard matrices in \Cref{sec:matrices}.
Next, we formally define the correlation of sequences
as well as perfect sequences, QT sequences, and Williamson-type sequences in \Cref{sec:sequences},
including a one-to-one correspondence between certain quaternionic perfect sequences and QT sequences
proven by Barrera Acevedo and Dietrich~\cite{BAD18}.
We then prove a one-to-one correspondence between QT sequences and Williamson-type sequences in \Cref{sec:qt-wt-equiv}.
Finally, we define notions of equivalence of QT sequences in \Cref{sec:equiv} that we later rely on to classify our QT sequence enumeration results.

\subsection{Quaternions}\label{sec:quat}

The quaternions may be viewed as a noncommutative extension of the complex numbers. 
Formally, the real algebra of quaternions is represented as
\begin{equation*}
\HH \eqdef \{\, \alpha + \beta \qi + \gamma \qj + \delta \qk : \alpha, \beta, \gamma, \delta \in \reals \,\} ,
\end{equation*}
where~$\qi$, $\qj$, and $\qk$ are the \emph{basic unit quaternions\/} satisfying the relations 
\[
\qi^2 = \qj^2 = \qk^2 = -1, \quad \qi \qj = \qk, \quad\text{and}\quad \qj \qi = -\qk .
\]
The number~$\alpha$ is the \emph{real part\/} of the quaternion~$\alpha + \beta \qi + \gamma \qj + \delta \qk$, and~$\beta \qi + \gamma \qj + \delta \qk$ its \emph{pure imaginary\/} part.

The algebra is endowed with a conjugation operation $x\mapsto x^*$ satisfying~$\alpha^\ast = \alpha$ for real $\alpha$
along with $\qi^\ast = -\qi$, $\qj^\ast = -\qj$, and  $\qk^\ast = -\qk$. The real part~$\Re(w)$ of a quaternion~$w$ can thus be expressed as~$\Re(w) = (w + w^*)/2$. 

Conjugation distributes over addition and satisfies~$(x_1 x_2)^\ast = x_2^\ast x_1^\ast$.
The norm~$\norm{x}$ of a quaternion~$x$ is given by
\[
\norm{x} \eqdef \sqrt{x x^\ast} \eq \sqrt{x^\ast x} , 
\]
and when~$x \neq 0$, its multiplicative inverse~$x^{-1}$ is given by~$x^{-1} \eqdef x^\ast / \norm{x}$.
The quaternions with norm~$1$ are the \emph{unit quaternions\/}, and we denote the set of unit quaternions by
\[
\qsphere^3 \eqdef \{\, w \in \HH : \norm{w} = 1 \,\} .
\]
The unit quaternions form a non-abelian group under the multiplication operation.
We denote the set of quaternionic square roots of~$-1$ by~$\qsphere^2$. That is,
\[
\qsphere^2 \eqdef \{\, w \in \HH : w^2=-1 \,\} ,
\]
and it can be shown that 
\[
\qsphere^2 \eq \{\, w \in \qsphere^3 : \Re(w)=0 \,\} . 
\]
So the elements of~$\qsphere^2$ are pure imaginary quaternions.

We define $q \eqdef (1+\qi+\qj+\qk)/2$. This is a unit quaternion that is a cube root of $-1$.
Three important sets of unit quaternions in this paper are
\[
\Q_8 \eqdef \{\pm1,\pm\qi, \pm\qj, \pm\qk\}, \quad \Q_+ \eqdef \Q_8 \cup \, q\Q_8 \,, \quad \text{and }  \quad \Q_{24} \eqdef \Q_+\cup \, q^2\Q_8 \,.
\]
The sets~$\Q_8$ and~$\Q_{24}$ are groups under multiplication, but~$\Q_+$ is not since $q^2=-q^*\notin\Q_+$.

Euler's formula $e^{\qi t}=\cos t+\qi\sin t$ generalizes to the quaternions via the power series expansion for the exponential function~\cite{Cho98-demoivre}. For $t \in \R$ and~$w\in \qsphere^2$, we have
\begin{equation*}
\label{eq-euler}
\e^{wt} = \cos t + w \sin t .
\end{equation*}

\subsection{Hadamard matrices}\label{sec:matrices}

A square matrix~$H$ of order $n$ is a \emph{(real) Hadamard matrix\/} if~$H$ has~$\pm1$ entries and $HH^\top = n \id_n$. In other words, all the entries of the matrix have unit magnitude and the rows of the matrix are mutually orthogonal. This implies that the columns of the matrix are also mutually orthogonal.

For a quaternionic matrix~$G$ of order~$n$, the matrix~$G^\adjoint$ denotes the conjugate transpose of~$G$, i.e., $(G^\adjoint)_{ij} \eqdef (G_{ji})^*$ for all 
indices~$i,j \in \set{0,1, \dotsc, n-1}$. 
The matrix~$G$ is a \emph{quaternionic Hadamard matrix\/} (QHM) if~$G$ has entries in $\qsphere^3$ and $G G^\adjoint = n\id_n$. The latter condition requires the rows of the matrix to be orthogonal with respect to a sesquilinear form on the~$\HH$-module~$\HH^n$ that is conjugate-linear in the second argument. We can show that this condition implies~$G^\adjoint G = n \id_n$, i.e., the columns of~$G$ are orthogonal with respect to a \emph{different\/} sesquilinear form that is conjugate-linear in the \emph{first\/} argument. Due to the non-commutativity of multiplication in~$\HH$, the use of the two different sesquilinear forms is important; orthogonality with respect to one does not necessarily imply orthogonality with respect to the other.

We may view the set of reals as a subset of quaternions with~$0$ pure imaginary part, and the set of complex numbers as a subset of quaternions generated by~$\set{1,\qi}$. So a complex Hadamard matrix may be viewed as a QHM with complex entries.

In the rest of the paper, we consider Hadamard matrices up to equivalence. Two real Hadamard matrices are considered to be equivalent if one can be obtained from the other by negating rows or columns, by interchanging rows, or by interchanging columns~\cite{McKay1979}.
We extend this equivalence to QT sequences: two QT sequences are said to be \emph{Hadamard equivalent}
if they generate equivalent quaternion-type Hadamard matrices.
Similarly, two quaternionic Hadamard matrices are said to be equivalent if one can be obtained from the other by multiplying rows \emph{on the left\/} by unit quaternions, by multiplying columns \emph{on the right\/} by unit quaternions, by interchanging rows, or by interchanging columns. Since the multiplication of quaternions is not commutative, the order of multiplication in these operations is important. The equivalence of two QHMs may be characterized as follows.
\begin{lemma}
\label{lem-equivalence}
Two QHMs~$G$, $G'$ are equivalent if and only if there are permutation matrices~$P_1$, $P_2$ and diagonal matrices~$D_1$, $D_2$ of unit quaternions such that~$G' = D_1 P_1 G P_2 D_2$.
\end{lemma}
An analogous lemma holds for real or complex Hadamard matrices. In the real case, the diagonals of~$D_1$ and~$D_2$ are~$\pm 1$, and in the complex case, they consist of unit complex numbers.

We say that a real, complex, or quaternionic Hadamard matrix is \emph{normalized\/} if the entries in the first row and the first column are all~`$1$'. Observe that every Hadamard matrix is equivalent to a normalized Hadamard matrix; we may multiply each row on the left by the conjugate of its first entry, and similarly multiply each column on the right by the conjugate of its first entry to achieve this. We refer to these operations as `dephasing'.

Given a quaternionic Hadamard matrix~$G$ and an algebra automorphism~$f$ of the quaternions, we may verify that the matrix~$f(G)$ defined by~$(f(G))_{ij} \eqdef f(G_{ij})$ for all indices~$0 \le i,j < n$ is also a quaternionic Hadamard matrix. It so happens that every automorphism of the quaternions corresponds to an \emph{inner\/} automorphism.
\begin{proposition}[Proposition~2.4.7, page~19, \cite{Rodman14-quaternion-linear-algebra}]
\label{prop-inner-aut}
Every automorphism of~$\HH$ is \emph{inner\/}, i.e., if~$f \colon \HH \rightarrow \HH$ is an automorphism, then there exists~$a \in \qsphere^2$ such that~$f(x) = a^* x a$, for all~$x \in \HH$. 
\end{proposition}
Note that the application of an inner automorphism to every entry of a matrix may be achieved by an equivalence operation; we may multiply each row on the left by~$a^*$ and each column on the right by~$a$. So the application of automorphisms results in equivalent QHMs.
\begin{corollary}
\label{cor:automorphism-equiv}
Applying an automorphism of\/ $\HH$ to each entry of a quaternionic Hadamard matrix
produces an equivalent quaternionic Hadamard matrix.
\end{corollary}

\subsection{Perfect, QT, and Williamson-type sequences}\label{sec:sequences}

Let~$A$ and~$B$ be sequences of length $n$, with $A = [a_0, \dotsc, a_{n-1}]$ and~$B = [b_0, \dotsc, b_{n-1}]$.
For a \emph{shift\/} $0 \le t < n$, the \emph{periodic crosscorrelation\/}~$R_{A,B}(t)$ of~$A$ and~$B$ is defined as
\[
R_{A,B}(t)\eqdef \sum_{r=0}^{n-1} a_r b^*_{r+t\bmod n} .
\]
We further define the \emph{right periodic autocorrelation\/}~$R_A^R(t)$ of~$A$ and the \emph{left periodic autocorrelation\/}~$R_A^L(t)$ of $A$ as
\begin{equation*}
    R_A^R(t)\eqdef\sum_{r=0}^{n-1}a_r a_{r+t\bmod n}^*,\quad\text{and}\quad
    R_A^L(t)\eqdef\sum_{r=0}^{n-1}a_r^*a_{r+t\bmod n},
\end{equation*}
respectively. 
A sequence~$S$ is called \emph{right perfect\/} if its right periodic~$t$-autocorrelation vanishes for all non-zero shifts~$t$, i.e., for all~$t$ with~$1 \leq t < n$, we have~$R^L_S(t) = 0$. Left perfection of a sequence is defined in a similar way. Kuznetsov~\cite{Kuznetsov2009} proved that a quaternion sequence is right perfect if and only if it is left perfect. Thus without loss of generality, we define the \emph{periodic autocorrelation\/}~$R_A(t)$ of $A$ as
\[
R_A(t)\eqdef R^R_A(t) = R_{A,A}(t)=\sum_{r=0}^{n-1}a_r a_{r+t\bmod n}^*,
\]
and call the sequence $A$ \emph{perfect\/} if its periodic autocorrelation values vanish for all non-zero $t$.
A perfect sequence over~$\HH$ is called a \emph{perfect quaternion sequence\/}. 

We refer to sequences with~$\{\pm1\}$ entries as \emph{binary\/} sequences.
Let $A$, $B$, $C$, and $D$ each be binary sequences of length~$n$.
Denote the Kronecker delta by~$\diracdelta_{0,t}$ (note $\diracdelta_{0,t}$ is $1$ if $t=0$ and $0$ otherwise).
We say that~$(A,B,C,D)$ is a quadruple of \emph{QT sequences} if for all $0\leq t < n$ we have
\begin{align}
\label{eq-ac}
R_{A}(t)+R_B(t)+R_C(t)+R_D(t) & \eq \diracdelta_{0,t} 4n , \\
\label{eq-cc1}
R_{A,B}(t)-R_{B,A}(t) + R_{C,D}(t)-R_{D,C}(t) & \eq 0 , \\
\label{eq-cc2}
R_{A,C}(t)-R_{C,A}(t) + R_{D,B}(t)-R_{B,D}(t) & \eq 0 , \\
\label{eq-cc3}
R_{A,D}(t)-R_{D,A}(t) + R_{B,C}(t)-R_{C,B}(t) & \eq 0.
\end{align}
We call~\eqref{eq-ac} the \emph{autocorrelation condition} and~\eqref{eq-cc1}--\eqref{eq-cc3} the \emph{crosscorrelation conditions}.
If the crosscorrelation conditions are replaced by the conditions
$R_{X,Y}(t)-R_{Y,X}(t)=0$ for all $X$, $Y\in\{A,B,C,D\}$
then the sequences in the quadruple are known as \emph{Williamson-type}.
QT sequences are equivalent to
quaternion-type Hadamard matrices with circulant blocks (as defined in the introduction)
and similarly Williamson-type sequences are equivalent to the first rows of
circulant Williamson-type matrices.
If a quadruple of Williamson-type sequences generate \emph{symmetric} circulant
block matrices then they are called \emph{Williamson sequences}~\cite{BKG20-williamson-conjecture}.
Note a sequence $[a_0,\dotsc,a_{n-1}]$ generates a circulant symmetric matrix when
$a_i=a_{n-i}$ for all $1\leq i<n$.

Barrera Acevedo and Dietrich~\cite{BAD18} showed that perfect quaternion sequences over~$\Q_+$ of length~$n$ are equivalent to
a certain kind of relative difference set in $C_n \times \Q_8$, where $C_n$ denotes the cyclic group of order~$n$.
A result of Schmidt~\cite{Schmidt1999} shows this kind of relative difference set in $C_n \times \Q_8$ is
equivalent to quaternion-type Hadamard matrices whose blocks are circulant
(but not necessarily pairwise amicable or symmetric).
Precisely, we have the following theorem.
\begin{theorem}[{cf.~\cite[Thm.~2.4]{BAD19}}]
\label{thm-bad-correspondence}
There is a one-to-one correspondence between quadruples of QT sequences of length\/~$n$ and perfect sequences of length\/~$n$ over\/~$\Q_+$.
\end{theorem}
This correspondence can be described by a map between the $r$th entry~$s_r$ of a perfect quaternion sequence~$S$ and the quadruple~$(a_r, b_r, c_r, d_r)$ of the~$r$th entries of the QT sequences~$(A,B,C,D)$. The following table specifies half of the mapping, with~$+1$ and~$-1$ abbreviated as~`$+$' and~`$-$', respectively.
\[
\begin{array}{ccccccccc} 
s_r & 1 & \qi & \qj & \qk & q & q\qi & q\qj & q\qk \\ 
\hline
a_r & - & + & + & + & + & + & + & + \\ 
b_r & - & - & + & - & - & + & + & - \\
c_r & - & - & - & + & - & - & + & + \\
d_r & - & + & - & - & - & + & - & + \\ 
\end{array}
\]
The remaining half of the mapping is given by the rule that if $(a_r, b_r, c_r, d_r)$ maps to $s_r$ then $(-a_r, -b_r, -c_r, -d_r)$ maps to $-s_r\mspace{0.7mu}$.
Due to this correspondence, searching for perfect quaternion sequences over the alphabet $\Q_+$ is equivalent to searching for QT sequences.

We stress the equivalence established in \Cref{thm-bad-correspondence} is \emph{not} between Williamson-type sequences and perfect $\Q_+$-sequences, but instead
between QT sequences (that are conceivably not pairwise amicable) and perfect $\Q_+$-sequences.
Schmidt explicitly warns that he does \emph{not}
require pairwise amicability of the block matrices in the result
used to prove \Cref{thm-bad-correspondence}:
\begin{quote}
We remark that this definition slightly differs from the usual one which requires that $XY^t = Y X^t$ holds for all
2-subsets $\{X, Y\}$ of $\{A, B, C, D\}$.~\cite{Schmidt1999}
\end{quote}
Thus, in order to complete an enumeration of perfect $\Q_+$-sequences via this equivalence one could unfortunately
\emph{not} use the stronger pairwise amicability condition
and would instead have to use the weaker quaternion-type conditions provided by~\eqref{eq:pcf}.
Fortunately, in the next section we show that the conditions~\eqref{eq:pcf} actually imply pairwise amicability,
thereby rendering the above distinction irrelevant and providing the following.
\begin{theorem}
There is a one-to-one correspondence between quadruples of Williamson-type sequences of length\/~$n$ and perfect sequences of length\/~$n$ over\/~$\Q_+$.
\end{theorem}

\subsection{The equivalence of QT sequences and Williamson-type sequences}\label{sec:qt-wt-equiv}

In order to show the equivalence of QT sequences and Williamson-type sequences we utilize the discrete Fourier transform (DFT)\@.
We denote by $\DFT_A(t)$ the $t$-th value of the DFT of a sequence $A$ of length~$n$,
i.e., $\DFT_A(t)=\sum_{s=0}^{n-1}a_i\exp(2\pi \qi st/n)$, and use
$\DFT(A)$ to denote the vector of DFT values, i.e., $\DFT(A)=[\DFT_A(0),\dotsc,\DFT_A(n-1)]$.
We also define the \emph{power spectral density} of $A$ by $\PSD_A(t) \coloneqq \DFT_A(t)\DFT_{A^*}(t) = \lvert\DFT_A(t)\rvert^2$
and write $\PSD_A$ for the vector of $\PSD_A(t)$ values for $t=0$ to $n-1$.
Here $A^*$ is the conjugate transpose operation applied to a sequence; the $k$th entry
of $A^*$ is $(A_{-k\bmod n})^*$ where the $\bmod$ operator gives the integer in $[0,n)$
congruent to its argument modulo~$n$.  The following equivalence is well-known
but for completeness we provide a proof.

\begin{proposition}[Power Spectral Density] \label{prop:psd}
$(A, B, C, D)$ is a quadruple of QT sequences of length $n$ if and only if
\[ \PSD_A(t) + \PSD_B(t) + \PSD_C(t) + \PSD_D(t) = 4n \quad\text{for all $t\in\mathbb{Z}$}. \]
\end{proposition}
\begin{proof}
The autocorrelation condition is equivalent to the convolution equation
\begin{equation}
A*A^* + B*B^* + C*C^* + D*D^* = [4n,0,\dotsc,0] \label{eq:conv}
\end{equation}
where $X*Y$ denotes sequence convolution, i.e.,
the convolution of the sequences $[x_0,\dotsc,x_{n-1}]$ and $[y_0,\dotsc,y_{n-1}]$
has as its $k$th entry $\sum_{i=0}^{n-1}x_iy_{k-i\bmod n}$.
From the DFT convolution theorem we have $\DFT_{X*X^*}(t)=\DFT_X(t)\DFT_{X^*}(t)=\PSD_X(t)$.
Applying the DFT to both sides of \eqref{eq:conv}, using the linearity of
the DFT, and the fact that the DFT of $[4n,0,\dotsc,0]$ is $[4n,\dotsc,4n]$, we have that
\eqref{eq:conv} implies
\[ \PSD_A + \PSD_B + \PSD_C + \PSD_D = [4n,\dotsc,4n] . \]
Applying the inverse DFT to this PSD equation shows it is equivalent
to \eqref{eq:conv}.
\end{proof}

Similarly, the following proposition shows that we can also apply the DFT the crosscorrelation
equations to derive an equivalent form of the equations in the frequency domain.
Define the \emph{cross power spectral density} by $\CPSD_{X,Y}(t)=\DFT_X(t)\DFT_Y(t)^*$ and
let $\CPSD_{X,Y}$ be the vector of CPSD values from $t=0$ to $n-1$.

\begin{proposition}[Cross Power Spectral Density]\label{prop:cpsd}
The crosscorrelation equation~\eqref{eq-cc1} holding for all integers\/~$t$ is equivalent to
the condition
\[ \CPSD_{A,B}(t) - \CPSD_{B,A}(t) + \CPSD_{C,D}(t) - \CPSD_{D,C}(t) = 0 \quad\text{for all $t\in\mathbb{Z}$}. \]
\end{proposition}
\begin{proof}
The vector of crosscorrelations of $X$ and $Y$ is equal to the convolution vector
$X*Y^*$, so the crosscorrelation condition~\eqref{eq-cc1} is equivalent to
\begin{equation}
 A*B^*-B*A^*+C*D^*-D*C^* = [0,\dotsc,0] . \label{eq:crossconv}
\end{equation}
By applying the DFT and using the DFT convolution theorem we have
$\DFT_{X*Y^*}(t)=\DFT_X(t)\DFT_{Y^*}(t)=\CPSD_{X,Y}(t)$, so
by applying the DFT to~\eqref{eq-cc1} and noting that the DFT
of $[0,\dotsc,0]$ is $[0,\dotsc,0]$, we have that~\eqref{eq:crossconv} is equivalent
to
\[ \CPSD_{A,B} - \CPSD_{B,A} + \CPSD_{C,D} - \CPSD_{D,C} = [0,\dotsc,0] . \qedhere \]
\end{proof}
\begin{remark}\label{rmk:cpsd}
A similar proof to that of \Cref{prop:cpsd} gives that~\eqref{eq-cc2} for all\/~$t$ is equivalent to
$\CPSD_{A,C} - \CPSD_{C,A} =  \CPSD_{B,D} - \CPSD_{D,B}$,
and~\eqref{eq-cc3} for all\/~$t$ is equivalent to
$\CPSD_{A,D} - \CPSD_{D,A} =  \CPSD_{C,B} - \CPSD_{B,C}$.
\end{remark}

Even though QT sequences and Williamson-type sequences are defined
differently, we now prove these two notions are actually equivalent.

\begin{theorem}
$(A,B,C,D)$ is a quadruple of QT sequences if and only if $(A,B,C,D)$ is a quadruple of
Williamson-type sequences.
\label{thm:qt-wt-equiv}
\end{theorem}

\begin{proof}
If $(A,B,C,D)$ is a set of Williamson-type sequences then $R_{X,Y}(t)-R_{Y,X}(t)=0$
for all integers~$t$ and $X,Y\in\{A,B,C,D\}$.  Adding together the conditions
for $(X,Y)=(A,B)$ and $(X,Y)=(C,D)$
one derives~\eqref{eq-cc1}.  The conditions~\eqref{eq-cc2} and~\eqref{eq-cc3}
are derived similarly.

Conversely, suppose $(A,B,C,D)$ is a set of QT sequences. Take $t$ to be a fixed integer,
and let $a\coloneqq\DFT_A(t)$, $b\coloneqq\DFT_B(t)$, $c\coloneqq\DFT_C(t)$, and
$d\coloneqq\DFT_D(t)$.  By \Cref{prop:cpsd} we have
\begin{equation} ab^* - ba^* + cd^* - dc^* = 0. \label{eq:negdet} \end{equation}
For $x\in\C$, write $x=r_x+i_x\qi$ with $r_x,i_x\in\R$.
Equivalently considering $x$ as a row vector in $\R^2$,
we write the determinant of the matrix $\bigl[\begin{smallmatrix}x\\y\end{smallmatrix}\bigr]=\bigl[\begin{smallmatrix}r_x&i_x\\r_y&i_y\end{smallmatrix}\bigr]$
as $\detii{x}{y}=r_xi_y-r_yi_x$, and note that $yx^*-xy^*=2\qi\detii{x}{y}$.
Thus \eqref{eq:negdet} is equivalent to $\detii{a}{b}+\detii{c}{d} = 0$ or
$\detii{a}{b}=\detii{d}{c}$.  From the crosscorrelation equations in the
frequency domain (cf.~\Cref{rmk:cpsd}) we have
\begin{equation} \detii{a}{b}=\detii{d}{c}, \quad \detii{a}{c}=\detii{b}{d}, \quad\text{and}\quad \detii{a}{d}=\detii{c}{b} . \label{eq:dets} \end{equation}
We now show this implies all these determinants are $0$.  Suppose not, and without loss
of generality that $\detii{a}{b}\neq0$.  Then $\{a,b\}$ form a basis of $\R^2$, so write
$c=\alpha a + \beta b$ and $d=\gamma a + \delta b$ for $\alpha$, $\beta$, $\gamma$, $\delta\in\R$.
Now
$\detii{d}{c}=\detii{\gamma a+\delta b}{\alpha a+\beta b}=\detii{\gamma a}{\alpha a+\beta b}+\detii{\delta b}{\alpha a+\beta b}=\detii{\gamma a}{\beta b}+\detii{\delta b}{\alpha a}=(\gamma\beta-\alpha\delta)\detii{a}{b}$,
$\detii{a}{c}=\beta\detii{a}{b}$, $\detii{b}{d}=\detii{b}{\gamma a + \delta b}=\detii{b}{\gamma a}=-\gamma\detii{a}{b}$,
$\detii{a}{d}=\delta\detii{a}{b}$, and $\detii{c}{b}=\alpha\detii{a}{b}$.  Using these expressions in~\eqref{eq:dets}, we have
\[ \gamma\beta-\alpha\delta = 1 , \quad \beta = - \gamma, \quad\text{and}\quad \delta = \alpha . \]
These imply $\alpha^2+\beta^2=-1$, a contradiction.  Thus it must be the case
that $\detii{a}{b}=\detii{a}{c}=\detii{a}{d}=\detii{b}{c}=\detii{b}{d}=\detii{c}{d}=0$.
From $\detii{a}{b}=0$ it follows $ab^*=ba^*$.  Since $a=\DFT_A(t)$ and $b=\DFT_B(t)$
and $t$ was arbitrary, it follows $\CPSD_{A,B}=\CPSD_{B,A}$.  Applying the inverse DFT, we have
$R_{A,B}(t)=R_{B,A}(t)$ for all~$t$.  Similarly, we have $R_{X,Y}(t)=R_{Y,X}(t)$ for all~$t$,
so $(A,B,C,D)$ is of Williamson-type.
\end{proof}

Despite the one-to-one correspondence between Williamson-type sequences and QT sequences, we still find it useful
to have separate terminology for both.  In the next section we
provide two different notions of equivalence of quadruples of QT sequences (one based 
on the definition of QT sequences and the other based on the definition of Williamson-type sequences)
and describe the usefulness of having both kinds of equivalence available.

\subsection{Sequence equivalence operations}\label{sec:equiv}

QT sequences have a number of equivalence operations that can be used to define equivalence classes 
that in general are not the same as the equivalence classes formed by the Hadamard equivalence described in \Cref{sec:matrices}.
Because we care about running exhaustive searches
for QT sequences up to equivalence, we use these equivalence operations to
effectively reduce the size of the search space without loss of generality.
Each operation acts on a quadruple of QT sequences $(A,B,C,D)$
and produces another quadruple of QT sequences.

\paragraph{QT equivalence operations}
\begin{itemize}
    \item NS: (Negate and Swap) Negate a single sequence of $A$, $B$, $C$, $D$ and swap a single pair of sequences of $A$, $B$, $C$, $D$.
    \item AN: (Alternating negation) Multiply every other element of $A$, $B$, $C$ and $D$ simultaneously by~$-1$ when the length~$n$ is even.
    \item DE: (Decimate) Apply an automorphism of the cyclic group $C_n$ to the indices of each sequence $A$, $B$, $C$, and $D$ simultaneously.
    \item CS: (Cyclic Shift) Cyclically shift all the entries of $A$, $B$, $C$, and $D$ simultaneously by any amount.
    \item DH: (Double Half Shift) Shift the entries of two of $A$, $B$, $C$, and $D$ by $n/2$ when the length~$n$ is even.
\end{itemize}
We call two quadruples of QT sequences \emph{QT equivalent} if they are equivalent
under these operations.  For convenience, we also define the following two additional equivalence operations. They both can be derived from
two applications of the NS operation.

\begin{itemize}
    \item DN: (Double Negate) Negate two sequences of $A$, $B$, $C$, or $D$.
    \item DS: (Double Swap) Swap two pairs of $A$, $B$, $C$, $D$.
\end{itemize}

The NS equivalence operation preserves the autocorrelation condition
since $R_X(t)=R_{-X}(t)$ for all~$t$, and it preserves the crosscorrelation conditions because
the action of applying a single negation and the action of a applying single swap
has the effect of negating a single side of each of
the three crosscorrelation conditions (as well as permuting two conditions
amongst themselves in the case of a swap) when written in the form
\begin{align*}
    R_{A,B}(t)-R_{B,A}(t)&\eq R_{D,C}(t)-R_{C,D}(t), \\
    R_{A,C}(t)-R_{C,A}(t)&\eq R_{B,D}(t)-R_{D,B}(t), \\
    R_{A,D}(t)-R_{D,A}(t)&\eq R_{C,B}(t)-R_{B,C}(t).
\end{align*}
Thus, the NS operation preserves the crosscorrelation conditions.
The DE, AN, and both shift operations are well-known to preserve
the sum of autocorrelation values~\cite{Lumsden2025}
and they also preserve the sum of crosscorrelation values.
For example, if $(-1)*X$ denotes the alternating negation of $X$,
then $R_{(-1)*A,(-1)*B}(t)=(-1)^t R_{A,B}(t)$ (for even length sequences)
and if $s(X)$ denotes a cyclic shift
of $X$ by 1 then $R_{s(A),s(B)}(t)=R_{A,B}(t)$ and
\[ R_{s^{n/2}(A),B}(t) = R_{A,s^{n/2}(B)}(t) = R_{A,B}(t\pm n/2) \]
when $n$ is even.
Sequence reversal can also be considered an equivalence operation, but in this case it
can be recovered by applying DE with the
automorphism $x\mapsto -x$ followed by a cyclic shift to the left.

Note that the above equivalence operations do \emph{not} make use of
the one-to-one correspondence between QT sequences and Williamson-type sequences
provided by \Cref{thm:qt-wt-equiv}.  Incorporating \Cref{thm:qt-wt-equiv},
we could replace the negate-and-swap operation
with a single-negate operation and a single-swap operation,
and the double half shift with a single half shift:
\begin{itemize}
    \item SN: (Single Negate) Negate one sequence of $A$, $B$, $C$, or $D$.
    \item SS: (Single Swap) Swap one pair of $A$, $B$, $C$, $D$.
    \item SH: (Single Half Shift) Shift the entries of one of $A$, $B$, $C$, $D$ by $n/2$ when~$n$ is even.
\end{itemize}
This is because the SN, SS, and SH operations preserve the pairwise amicability
condition $R_{X,Y}(t)=R_{Y,X}(t)$ satisfied by Williamson-type sequences.
Consequently, we call two quadruples \emph{Williamson-type equivalent}
if they are equivalent under the
QT equivalence operations with the NS operation replaced by SN and SS
and the DH operation replaced by SH\@.

The primary reason why we define QT equivalence
in a different way from Williamson-type equivalence is because QT equivalence
relates better to Hadamard equivalence---in \Cref{thm:equivalence}
we show that QT equivalence implies Hadamard equivalence, so to enumerate
all QT sequences up to Hadamard equivalence one can enumerate all sequences
up to QT equivalence and then filter by Hadamard equivalence.  On the other
hand, Williamson-type equivalence does \emph{not} imply Hadamard equivalence
in general.
Thus, an enumeration of QT sequences
up to Williamson-type equivalence is not guaranteed to produce an enumeration
of QT sequences up to Hadamard equivalence.
For simplicity, the algorithm we present will enumerate QT sequences up to
Williamson-type equivalence, but
in Section~\ref{sec:equiv_filtering}
we describe how to also ensure a complete
enumeration up to both QT equivalence and Hadamard equivalence.

\begin{theorem}\label{thm:equivalence}
QT sequences that are QT equivalent are also Hadamard equivalent.
\end{theorem}

\begin{proof}
We need to show that each equivalence operation
NS, AN, DE, CS, and DH preserves Hadamard equivalence.
Barrera Acevedo, Cath\'{a}in, and Dietrich have shown that if two QT sequences are
equivalent under NS equivalence then they are equivalent under Hadamard equivalence~\cite[Prop.~4.4(a,b)]{BarreraAcevedo2019},
so we have four remaining operations to consider.
In each case we show that application of the operation
to a quadruple of QT sequences produces a new quadruple which
is Hadamard equivalent to the original.  This is accomplished
by showing that each equivalence operation (considered as an operation
on quaternion-type Hadamard matrices) can be constructed out of the
operations of negating or permuting rows or columns.

For the AN operation, suppose $n$ is even.  Consider the act of
negating every second row and every second column of a quaternion-type
Hadamard matrix with circulant blocks.  Because $n$ is even, the resulting
Hadamard matrix still has circulant blocks and the defining sequences of each
circulant block have had AN applied to them.  Thus, applying AN to a QT sequence quadruple
produces a Hadamard equivalent quadruple.

The DE operation applies the same permutation $\phi$ over $\{0,\dotsc,n-1\}$ to
the indices of all entries of a quadruple.
Consider the act of using $\phi$ to permute both the columns and rows
with indices in $[0,n)$, $[n,2n)$, $[2n,3n)$, and $[3n,4n)$ of a quaternion-type
Hadamard matrix with circulant blocks.  (Here $\phi$ must be extended
to be over $\{0,\dotsc,4n-1\}$, e.g., via $\phi(i)=\phi(i-n)+n$ for $n\leq i<4n$.)
Because the same permutation was applied to both the rows and columns of each
circulant block, the resulting Hadamard matrix still has circulant blocks
(but their defining sequences have been permuted by $\phi$ and thus have had
DE applied to them).  Thus, the DE operation applied to a QT sequence quadruple
produces a Hadamard equivalent quadruple.

The CS operation with a shift of $c$ applies $s^c(i)=(i+c)\bmod n$ to
the indices of all entries of a quadruple.
Consider the act of using $s^c$ to circularly shift the columns with indices
in $[0,n)$, $[n,2n)$, $[2n,3n)$, and $[3n,4n)$ of a quaternion-type
Hadamard matrix with circulant blocks.
Because a circular shift applied to the columns of a circulant matrix produces
another circulant matrix, the resulting Hadamard matrix still has circulant
blocks and those blocks have had the CS operation applied to their defining
sequences.  Thus, the CS operation applied to a QT sequence quadruple
produces a Hadamard equivalent quadruple.

The DH operation applies a half-shift of $s^{n/2}$ to two sequences in a
quadruple of QT sequences (without loss of generality, suppose DH is applied
to the first two sequences in the quadruples).
If $X$ is a circulant matrix, denote by $X'$ the
columns of $X$ circularly shifted by an offset of $n/2$.  If $A$, $B$, $C$, and $D$ denote
the circulant blocks of a quaternion-type Hadamard matrix,
after applying $s^{n/2}$ to the columns with indices
in $[0,n)$ and $[n,2n)$, the resulting Hadamard matrix has
the form
\[ \begin{bmatrix}
A' & B' & C & D \\
-B' & A' & -D & C \\
-C' & D' & A & -B \\
-D' & -C' & B & A
\end{bmatrix} . \]
Now, applying $s^{n/2}$ to the rows with indices in $[2n,3n)$ and $[3n,4n)$ we obtain
\[ \begin{bmatrix}
A' & B' & C & D \\
-B' & A' & -D & C \\
-C & D & A' & -B' \\
-D & -C & B' & A'
\end{bmatrix} \]
which is the Hadamard matrix generated by applying DH to the QT quadruple corresponding
to $(A,B,C,D)$.
\end{proof}

\section{Exhaustive searches for QT sequences}\label{sec:search}

In this section, we first describe previous work that has been done on exhaustive search for perfect quaternion sequences, QT sequences, and related problems (see \Cref{sec:prev_work}).
We then present an algorithm that we used to extend the search for QT sequences to longer lengths in \Cref{sec:algorithm},
and explain the implementation and optimization of the algorithm in \Cref{sec:implementation}.
After the enumeration is complete, we perform a postprocessing of the results by removing QT sequences
that are equivalent up to QT, Williamson-type, and Hadamard equivalence (see \Cref{sec:equiv_filtering}).
Finally, we summarize the results found by our enumeration in \Cref{sec:results}.

\subsection{Previous work}\label{sec:prev_work}

In 2010, Kuznetsov~\cite{Kuznetsov2010} studied perfect sequences over~$\Q_{24}$
and ran searches for perfect $\Q_{24}$-sequences equivalent to sequences the form
\[ [z,1,\qj,x_0,x_1\dotsc,x_t,x_t,\dotsc,x_1,x_0,\qj,1] \]
where $z\in\Q_{24}$ and $x_0$, $\dotsc$, $x_t\in\Q_8$.  Note such sequences are symmetric
and this assumption allowed Kuznetsov's search to reach length 23 using a desktop computer for an unspecified
amount of time.
He discovered sequences of lengths 5, 7, 9, 11, 13, 17, 19, and~23,
and used these sequences to create a perfect $\Q_{24}$-sequence of length over~$5$ billion.
In~2012, Barrera Acevedo and Hall~\cite{AcevedoHall2012} constructed perfect sequences over the alphabet $\{\pm1, \pm\qi, \qj\}$
for all lengths of the form $q + 1$ where $q \equiv 1 \pmod 4$, and~$q$ is a prime power.
In~2017, Blake \cite{BT17} ran extensive searches for perfect sequences over~$\Q_8$.

In~2019, Barrera Acevedo, Cath{\'a}in, and Dietrich~\cite{BarreraAcevedo2019}
ran an exhaustive search to find cocyclic Hadamard matrices of prime order~$p \leq 13$.
One part of their results include QT sequences directly, though their search also
includes quadruples of other types that differ from QT sequences.
Their algorithm enumerates the possibilities for three out of the four sequences, and then uses
the first three sequences to generate all possibilities for the fourth sequence.

Also in~2019, Bright, Kotsireas, and Ganesh ran a search to find Williamson sequences.
Their algorithm separates the four sequences into two pairs, enumerates all the possibilities for these pairs,
and then combines the results from the enumeration of the pairs to generate all quadruples of Williamson sequences.
We adapt this algorithm to work for QT sequences
and show that this approach is more efficient than an approach that starts by enumerating
the possibilities for three out of the four sequences.

\subsection{An enumeration algorithm for QT sequences}\label{sec:algorithm}

We would like to search for a set of four binary sequences satisfying the autocorrelation and crosscorrelation conditions for QT sequences.
The first step of the algorithm we present (Algorithm~\ref{alg:qts} below) is to look at the ``rowsums'' of the sequences.
For a binary sequence $A=[a_0,\dotsc,a_{n-1}]$, we define $r_A \coloneqq \sum_{i=0}^{n-1} a_i$ to be the \emph{rowsum\/} of $A$.
The following lemma is classic (cf.~\cite[eq.~15]{Williamson1944}).

\begin{lemma}[Rowsum Decomposition]
If $A$, $B$, $C$, $D$ are a quadruple of QT sequences of length $n$, then
\[ r_A^2 + r_B^2+r_C^2+r_D^2 = 4n \, . \]
\end{lemma}

This lemma means that for a given length $n$, we can limit our search to sequences whose rowsums satisfy
a specific rowsum decomposition of~$4n$.
We compute all possible decompositions of~$4n$ as the sum of
four integer squares and enumerate the QT sequences for each possible decomposition.

The next part of the algorithm is to pick one decomposition and separate the four sequences into two pairs.
Each pair can be used to check the part of the conditions they are required to satisfy as a whole.
For example, for all shifts $t\neq0$ the sum of the autocorrelations of the first pair is required to be the negation of the sum of the autocorrelations of the second pair (since they must add up to zero).
We can also apply this reasoning to the crosscorrelation conditions,
since once the first pair (say $(A,B)$) is known, \Cref{thm:qt-wt-equiv} implies
that $R_{A,B}(t)=R_{B,A}(t)$ for all~$t$.  Thus, if $R_{A,B}(t)\neq R_{B,A}(t)$ for some~$t$
then the pair $(A,B)$ cannot be part of a QT sequence and can be ignored.

Therefore, we enumerate the set of pairs of sequences
with rowsums $(r_A,r_B)$, separately enumerate the pairs of sequences with rowsums $(r_C,r_D)$, compute their auto- and crosscorrelations,
and check if there are pairs of sequences from the first and second enumerations that
have autocorrelations that exactly cancel out and crosscorrelations that do not change under a swap.
If there are, then we would have found a quadruple of binary sequences satisfying
\eqref{eq-ac} and~\eqref{eq-cc1} and therefore a potential quadruple of QT sequences.
To ensure the four sequences are in fact a quadruple of QT sequences, we need to check if they satisfy the other two crosscorrelation conditions
\eqref{eq-cc2} and \eqref{eq-cc3}.

The pseudocode for this algorithm is given in Algorithm~\ref{alg:qts}.  Intuitively,
the reason this algorithm is faster than enumerating triples of QT sequences
is because enumerating pairs with rowsums $(r_A,r_B)$ and $(r_C,r_D)$ is faster than
enumerating triples with rowsums $(r_A,r_B,r_C)$.  Moreover,
the step that finds matches in the correlation values (step~\ref{step-matching}) is fast in practice
when implemented efficiently as described in \Cref{sec:implementation}.
In order to shrink the search space and improve efficiency,
Algorithm~\ref{alg:qts} enumerates QT sequences up to Williamson-type equivalence
by using only certain decompositions of $4n$ into a sum of four squares that
can be assumed without loss of generality (see \Cref{prop:wt-rowsums} for justification).
Although Algorithm~\ref{alg:qts} is guaranteed to enumerate
all QT sequences of a given order up to Williamson-type equivalence, in general it generates the
multiple quadruples in the same equivalence class.  Thus, after running Algorithm~\ref{alg:qts}
we run a postprocessing step (described in Section~\ref{sec:equiv_filtering})
to remove duplicate sequences up to equivalence.

\begin{algorithm}
\begin{algorithmic}[1]
\STATE $\cL\coloneqq\emptyset$ \com Set that will contain the QT sequences found
\STATE \com Decompose $4n$ as a sum of four nonnegative integer squares \\
$\mathcal{D}\coloneqq\{\,(w,x,y,z)\in\N^4:w^2+x^2+y^2+z^2=4n\text{ and }w\geq x\geq y\geq z\,\}$ \label{step-decompose}
\FOR{$(w,x,y,z) \in \mathcal{D}$} 
\STATE Construct $\mathcal W$, $\mathcal X$, $\mathcal Y$, $\mathcal Z$, the sets of binary sequences with rowsums equal to $w$, $x$, $y$, $z$, respectively \label{step-construct-seqs}
\STATE $\cL_1\coloneqq \mathcal W\times\mathcal Z$ and $\cL_2\coloneqq \mathcal X\times\mathcal Y$ \label{step-gen}
\STATE \com Compute autocorrelations of each candidate pair in $\cL_1$ \\
Compute $[R_W(t)+R_Z(t):1\leq t\leq n/2]$ for all $(W,Z)\in\cL_1$ with $R_{W,Z}(t)=R_{Z,W}(t)$ for all $1\leq t\leq n/2$ and store in a list $C_1$ \label{step-store1}
\STATE \com Compute \emph{negated} autocorrelations of each candidate pair in $\cL_2$ \\
Compute $[-R_X(t)-R_Y(t):1\leq t\leq n/2]$ for all $(X,Y)\in\cL_2$ with $R_{X,Y}(t)=R_{Y,X}(t)$ for all $1\leq t\leq n/2$ and store in a list $C_2$ \label{step-store2} \\
\com Matching vectors in $C_1$ and $C_2$ provide solutions of~\eqref{eq-ac} and~\eqref{eq-cc3}
\FOR{each list of values $v$ appearing in both~$C_1$ and~$C_2$} \label{step-matching}
\STATE Get the set~$\cS$ of all the quadruples of sequences generating $v$ in both lists~$C_1$ and~$C_2$
\FOR{each quadruple $(W,X,Y,Z)$ in $\cS$}
\IF {$(W,X,Y,Z)$ is a quadruple of QT sequences}
\STATE Add $(W,X,Y,Z)$ to $\cL$ \label{step-add-seq}
\ENDIF
\ENDFOR
\ENDFOR \label{step-end-for}
\ENDFOR
\RETURN $\cL$
\end{algorithmic}
\caption{Algorithm for enumerating QT sequences of a given order~$n$
up to Williamson-type equivalence.
}
\label{alg:qts}
\end{algorithm}

\subsection{Implementation and optimizations}\label{sec:implementation}

One of the most important optimizations to speed up our search for QT sequences is to quickly remove unwanted sequences from consideration.  A well-known ``filtering''
technique enabling the discarding of many sequences is based on the Discrete Fourier Transform (DFT)
defined in \Cref{sec:qt-wt-equiv}.
\Cref{prop:psd} says that if $(A, B, C, D)$ is a quadruple of QT sequences of length $n$, then
\[ \PSD_A(t) + \PSD_B(t) + \PSD_C(t) + \PSD_D(t) = 4n \quad\text{for all $t\in\mathbb{Z}$}. \]
This property was already used by Williamson~\cite[eq.~14]{Williamson1944}, except
written in the form $\DFT_A(t)^2+\DFT_B(t)^2+\DFT_C(t)^2+\DFT_D(t)^2=4n$ because Williamson
assumed the symmetry of sequences (and if $X$ is conjugate symmetric then $\PSD_X(t)=\DFT_X(t)^2$).
\begin{corollary}\label{cor:PSD}
If\/ $(A, B, C, D)$ is a quadruple of QT sequences of length $n$, then
\[ \PSD_A(t)+\PSD_B(t) \leq 4n \text{ and } \PSD_C(t)+\PSD_D(t) \leq 4n . \]
\end{corollary}
\begin{proof}
Follows from Proposition~\ref{prop:psd} since PSDs are nonnegative.
\end{proof}

\Cref{cor:PSD} is applied as a filter on step~\ref{step-gen} of \Cref{alg:qts}; if any pair
has a PSD value sum that is larger than $4n$ it is not included in $\cL_1$ or $\cL_2$.
Similarly, on step~\ref{step-construct-seqs} any sequence that has a PSD value larger than $4n$
is not included in $\mathcal W$, $\mathcal X$, $\mathcal Y$, or $\mathcal Z$.
In our implementation, the PSD values were computed using FFTW~\cite{Frigo2005}.

For step~\ref{step-decompose}, finding all decompositions of $4n$ into a sum of four integer squares can be done with brute force, but not all decompositions are necessary.
One immediate optimization is to only consider decompositions $(w,x,y,z)$ where the parities of $w$, $x$, $y$, and~$z$ match the parity
of $n$, because if $X\in\{\pm1\}^n$ then $r_X\equiv n\pmod2$.  The following proposition gives justification
for why it is sufficient to only consider nonnegative decompositions in sorted order.
\begin{proposition}\label{prop:wt-rowsums}
Every quadruple of QT sequences is Williamson-type equivalent to a quadruple whose rowsums
are of the form $(w,x,y,z)$ where $w\geq x\geq y\geq z\geq 0$.
\end{proposition}
\begin{proof}
By applying the SN operation, every QT sequence can be made equivalent to one in which all
rowsums are nonnegative.  By then applying the SS operation those rowsums can then be sorted
in descending order while preserving the nonnegativity of the rowsums.  Thus, every QT sequence
is equivalent to one with nonnegative rowsums sorted in descending order.
\end{proof}
The following proposition
provides a variant of \Cref{prop:wt-rowsums} but in terms of QT equivalence
instead of Williamson-type equivalence.
\begin{proposition}\label{prop:rowsum}
Every quadruple of QT sequences is QT equivalent to a quadruple whose rowsums are of the form
$(\pm w,x,y,z)$ where $w\geq x\geq y\geq z\geq 0$.  Furthermore, if $w=x$, $x=y$, $y=z$, or $z=0$, then
the sign of the first rowsum can additionally be taken to be positive.
\end{proposition}
\begin{proof}
By applying appropriate DN and DS operations, every QT sequence can be made equivalent to one in which the last three rowsums are all nonnegative
and sorted in decreasing order and the first rowsum is the largest or second largest rowsum in absolute value.  By applying the NS operation, one can also sort the first two rowsums by absolute value (by swapping the first two sequences, if necessary, then negating them), so any QT sequence is equivalent to one with a rowsum decomposition of the form
$(\pm w,x,y,z)$ where $w\geq x\geq y\geq z\geq 0$.
Furthermore, in cases where $z=0$ one can apply the DN operation to make the first rowsum positive while
keeping the rowsums in sorted order (by negating the first and last sequences).
Similarly, if at least two values in $\{w,x,y,z\}$ are equal, one can apply the NS operation to make the first rowsum positive
while keeping the rowsums in sorted order (by negating the first sequence then swapping the sequences with equal rowsums).
So if $z=0$ or $w=x$ or $x=y$ or $y=z$, the decomposition $(-w,x,y,z)$ can be avoided.
\end{proof}

In case one wants to modify \Cref{alg:qts} to enumerate QT sequences up to
QT equivalence (instead of Williamson-type equivalence),
an efficient way to do this would be to modify step~\ref{step-add-seq}:
after adding $(W,X,Y,Z)$ to $\cL$, also add $(-W,X,Y,Z)$ to $\cL$,
though this would only need to be done
in cases when $z\neq0$, $z\neq y$, $y\neq x$, and $x\neq w$.
\Cref{prop:rowsum} guarantees in cases when
$w=x$ or $x=y$ or $y=z$ or $z=0$ that \Cref{alg:qts} generates
all QT sequences up to QT equivalence.

For step~\ref{step-gen}, we choose to generate the pairs so that the largest and smallest rowsums $w$ and $z$ belong to the same pair.
We attempted all possible pairing combinations and found this to be the most efficient.

By \Cref{thm:qt-wt-equiv}, all QT sequences satisfy $R_{X,Y}(t)=R_{Y,X}(t)$ and $R_{W,Z}(t)=R_{Z,W}(t)$,
so any pairs not satisfying these amicability conditions are not considered in steps~\ref{step-store1} and~\ref{step-store2}.
This filter dramatically reduces the number of pairs added to the lists $C_1$ and $C_2$.
For example, in length 20, over 16~billion pairs are generated without this filter,
while only 625,896 are generated using it---reducing the number of pairs generated by a factor of about $26{,}000$.

For steps~\ref{step-store1} and~\ref{step-store2}, the reason one does not need to consider the correlation values at $t=0$
is because the autocorrelation at $t=0$ is the rowsum of the sequence (which has already been fixed by construction)
and the crosscorrelation difference $R_{X,Y}(0)-R_{Y,X}(0)$ is always zero for real sequences $X$ and $Y$.
Similarly, the reason the correlation values for $n/2<t<n$ do not need to be considered is because of the symmetries
$R_X(t)=R_X(n-t)$ and $R_{X,Y}(t)=R_{Y,X}(n-t)$ for real sequences $X$ and $Y$.

For steps~\ref{step-store1} and~\ref{step-store2}, we store the list of autocorrelation values generated by one pair as one line in a file,
and we store the indices of the sequences used to generate those values at the end of the line.
This way, we can quickly determine which pair the correlation values on each line was derived from, even after the lines of the
files have been permuted.

In order to implement step~\ref{step-matching} efficiently, we sort the lines of the files containing the correlation values by lexicographic order.
This way, we iterate through the files simultaneously using a linear scan to efficiently find all matching correlation values from the
file containing the list $C_1$ and the file containing the list $C_2$.
We test all matches and keep the ones that satisfy the three crosscorrelation conditions.

Computing the correlation vectors in steps~\ref{step-store1} and~\ref{step-store2} using
the straightforward approach uses $O(n^2)$ arithmetic operations.  An optimization
is to match PSD values instead of matching correlation values.  That is,
instead of finding matches where
\[[R_W(t)+R_Z(t):1\leq t\leq n/2]=[-R_X(t)-R_Y(t):1\leq t\leq n/2],\]
by \Cref{prop:psd} it is sufficient to find matches where
\[[\PSD_W(t)+\PSD_Z(t)\mspace{-1.5mu}:\mspace{-1.5mu}1\leq t\leq n/2]=[4n-\PSD_X(t)-\PSD_Y(t)\mspace{-1.5mu}:\mspace{-1.5mu}1\leq t\leq n/2].\]
The PSD values were already computed as a part of the filtering used in step~\ref{step-construct-seqs} and
these values can be cached, so that they do not need to be recomputed.
In general PSD values are real and not integers, so
in order to write the value of $\PSD_W(t)+\PSD_Z(t)$ and $4n-\PSD_X(t)-\PSD_Y(t)$ to a file we round
them to the nearest integer.
This has the downside of potentially introducing spurious matches.  However, in practice this is
not a problem and once a match is found it can be checked easily if~\eqref{eq-ac} holds or not.

Similarly, instead of verifying the conditions
$R_{W,Z}(t)=R_{Z,W}(t)$ and $R_{X,Y}(t)=R_{Y,X}(t)$ in steps~\ref{step-store1} and~\ref{step-store2}
it is sufficient to work in the frequency domain and instead verify
\[ \CPSD_{W,Z}(t)=\CPSD_{Z,W}(t) \quad\text{and}\quad \CPSD_{X,Y}(t)=\CPSD_{Y,X}(t) . \]
This is beneficial as the calculation of the CPSD values can be optimized
by caching the values $\DFT_W(t)$, $\DFT_X(t)$, $\DFT_Y(t)$,
$\DFT_Z(t)$ for $1\leq t\leq n/2$ in step~\ref{step-construct-seqs}.
In order to account for floating-point roundoff error
we filter pairs $(W,Z)$ for which there is some $t$ with
\[ \lvert\CPSD_{W,Z}(t)-\CPSD_{Z,W}(t)\rvert > 10^{-4} , \]
since FFTW will compute
the $\CPSD$ values many orders of magnitude more accurately than $10^{-4}$.
The $\CPSD_{W,Z}(t)-\CPSD_{Z,W}(t)$ values have no real component, because
they are of the form $ab^*-ba^*$ for complex $a$, $b$ and as noted
in \Cref{sec:qt-wt-equiv} such numbers have no real component.

\subsection{Equivalence filtering}\label{sec:equiv_filtering}

After enumerating all QT sequences of length~$n$ using \Cref{alg:qts}, we classify the sequences
up to Hadamard equivalence (as described in \Cref{sec:matrices}) as well as
QT and Williamson-type equivalence (as described in \Cref{sec:equiv}).
\Cref{alg:qts} will generate all QT sequences of length~$n$ whose rowsums are in a form
given by \Cref{prop:wt-rowsums}.  Because \Cref{prop:wt-rowsums} was proven only using the SS and SN operations,
\Cref{alg:qts} will produce at least one quadruple
of QT sequences from each Williamson-type equivalence class.
However, if we also want to enumerate QT sequences up to Hadamard equivalence, we need
to enumerate up to QT equivalence instead, because then \Cref{thm:equivalence}
implies the list will be complete up to Hadamard equivalence.
In \Cref{sec:implementation}, we describe how to modify \Cref{alg:qts} to enumerate
sequences up to QT equivalence (see the remark after \Cref{prop:rowsum}).
We now describe how to filter out any duplicates up to QT equivalence.

As a first step in the equivalence filtering process
we take the output of \Cref{alg:qts} and filter its output to
keep exactly one ``canonical'' representative of every NS equivalence class.  The canonical representative
of an equivalence class will be
its lexicographically minimal quadruple when the sequences in a quadruple of QT sequences are concatenated.
In order to find the lexicographic minimum, we first use the NS operation to make the initial entry
of each sequence in the quadruple $-1$.  Next, we use a DS operation to make the lexicographically
least sequence appear first in the quadruple and the lexicographically second-least sequence appear
second in the quadruple.  If all sequences now appear in lexicographically ascending order, the
quadruple of QT sequences is canonical.  Otherwise, the last two sequences appear out of order.
If there are two identical sequences in the quadruple then
we use the DS operation to swap the identical sequences and swap the third and fourth sequences, sorting the sequences in
the quadruple in ascending order and making the quadruple canonical.  If there are not two identical
sequences in the quadruple we use the NS operation to swap the third and fourth sequences
and finally negate the fourth sequence to make the quadruple canonical.

Using the above process we compute the canonical representative of each quadruple of QT sequences found by \Cref{alg:qts}
and discard quadruples whose canonical form have been seen before, thereby
producing a complete list of QT sequences up to NS equivalence.
Now, we want to filter this list up to the QT equivalences of \Cref{sec:equiv}.  To do this, we
adapt a filtering algorithm described by Lumsden, Kotsireas, and Bright~\cite[Alg.~7]{Lumsden2025}.  First, consider
the symmetry group $S$ generated by the AN, DE, CS, and DH operations (letting~$\varphi$
denote the Euler totient function, this group is of size $2^6n \mspace{1.5mu} \varphi(n)$ when $n$ is even
and not divisible by 4, of size $2^4n\mspace{1.5mu}\varphi(n)$ when $n$ is divisible by 4, and
of size $n \mspace{1.5mu} \varphi(n)$ when $n$ is odd).
To find the lexicographically minimal representative of a quadruple
of QT sequences $(W,X,Y,Z)$ we apply all operations in $S$ to $(W,X,Y,Z)$, thereby forming $|S|$ QT sequences,
and use the above sorting process to lexicographically minimize each QT sequence using the NS operation.
The global lexicographic minimum of these $|S|$ sequences becomes the minimal representative of the equivalence class
generated by the operations in $\{\text{NS},\text{AN},\text{DE},\text{CS},\text{DH}\}$.

We also filter the list up to Hadamard equivalence.
By \Cref{thm:equivalence}, we know that the filtered list computed above
contains at least one representative from every Hadamard equivalence class.
However,
more filtering may be necessary to
remove duplicates up to Hadamard equivalence.
Starting from the complete list of quadruples
up to QT equivalence we convert each quadruple of QT sequences into a Hadamard matrix and reduce the matrix
to its graph representation as described by McKay~\cite{McKay1979}.  Then we use the program \textsc{nauty}~\cite{McKay2014}
to convert each graph into a canonical form.  Two graphs have the same canonical form if and only if they are
isomorphic, so we filter any QT sequence that results in a graph whose canonical form has been seen before.

The process for filtering up to QT equivalence described above can also be used to filter up to Williamson-type equivalence,
except we use single-swap (SS) and single-negate (SN) operations instead
of negate-and-swap (NS) operations, and similarly use the single half shift (SH)
instead of the double half shift (DS).  In this case the canonical representative found
with the $\{\text{SS},\text{SN}\}$ operations will always have each sequence sorted in lexicographically
decreasing order and the initial entry of each sequence will be $-1$.

In practice, to enumerate QT sequences up to each of Williamson-type equivalence, QT equivalence, and Hadamard equivalence we
do the following.
\begin{enumerate}
\item Enumerate sequences up to Williamson-type equivalence using \Cref{alg:qts} and place them in a list.
\item Filter the list by Williamson-type equivalence to remove duplicate quadruples up to equivalence.
\item For each $(W,X,Y,Z)$ in the list, add $(-W,X,Y,Z)$ to the list.
Following this, if $n$ is even, for each $(W,X,Y,Z)$ in the list, add $(s^{n/2}(W),X,Y,Z)$
to the list.
\item Filter the list up to QT equivalence.
This list will be exhaustive up to QT equivalence because
each NS equivalence class splits
into at most two classes under $\{\text{SS},\text{SN}\}$ equivalence, and
when it does split $(W,X,Y,Z)$ and $(-W,X,Y,Z)$ will be in separate classes.
Similarly, each DH equivalence
class splits into at most two classes under SH equivalence, and when it does
split $(W,X,Y,Z)$ and $(s^{n/2}(W),X,Y,Z)$ will be in separate classes.
\item Filter the list up to Hadamard equivalence.
\end{enumerate}

\subsection{Results}\label{sec:results}

We implemented the algorithm in Rust 1.91.0 and Bash shell scripts.
The algorithm was run on a machine with an AMD EPYC CPU running at 2.7~GHz and configured
to use no more than 4 GiB of RAM\@.
Sorting was done using the GNU sort utility.
The code and the results for the algorithm can be found at \url{https://github.com/colinotp/quaternion-sequences}.
A summary of the results found by the algorithm is provided in Table~\ref{table:results}.

\begin{table}
\caption{\label{table:results} Table of results for Algorithm~\ref{alg:qts} by length $n$.
$W_\text{equ}$, $Q_\text{equ}$, and $H_\text{equ}$ indicate the total number of QT sequences found
up to Williamson-type, QT, and Hadamard equivalence, respectively.
``Time'' is the total enumeration runtime in seconds,
``Pairs'' is the total number of pairs generated in step~\ref{step-gen} (generating these is
the bottleneck of the algorithm),
and ``Space'' is the amount of disk space used in MiB\@.}
\centering
\begin{tabular}{rrrrrrr}
$n$ & $W_{\text{equ}}$ & $Q_{\text{equ}}$ & $H_{\text{equ}}$ & Time (s) & Pairs & Space \\
\hline
1 & 1 & 1 & 1 & 0.16 & 2 & 0.0 \\
2 & 1 & 1 & 1 & 0.08 & 4 & 0.0 \\
3 & 1 & 1 & 1 & 0.08 & 6 & 0.0 \\
4 & 2 & 3 & 2 & 0.19 & 46 & 0.0 \\
5 & 1 & 1 & 1 & 0.08 & 20 & 0.0 \\
6 & 1 & 1 & 1 & 0.16 & 48 & 0.0 \\
7 & 2 & 3 & 3 & 0.15 & 182 & 0.0 \\
8 & 3 & 4 & 3 & 0.08 & 384 & 0.0 \\
9 & 4 & 7 & 7 & 0.31 & 999 & 0.0 \\
10 & 2 & 4 & 2 & 0.23 & 770 & 0.0 \\
11 & 1 & 2 & 2 & 0.15 & 715 & 0.0 \\
12 & 5 & 10 & 6 & 0.80 & 6288 & 0.2 \\
13 & 4 & 6 & 6 & 3.14 & 8216 & 0.4 \\
14 & 5 & 12 & 10 & 3.15 & 5334 & 0.3 \\
15 & 4 & 7 & 7 & 23.46 & 15732 & 1.0 \\
16 & 18 & 44 & 19 & 121.20 & 206480 & 7.9 \\
17 & 4 & 5 & 5 & 361.36 & 39372 & 3.6 \\
18 & 28 & 88 & 82 & 2001.79 & 275184 & 17.5 \\
19 & 6 & 11 & 11 & 7753.85 & 200526 & 22.0 \\
20 & 24 & 84 & 54 & 17632.48 & 625896 & 46.1 \\
21 & 7 & 13 & 13 & 141267.63 & 561519 & 86.1 \\
\end{tabular}
\end{table}

In the appendix we provide the complete enumeration of all 124 sets of
QT sequences (up to Williamson-type equivalence) for lengths $n\leq21$
found by our implementation.
By combining the enumerations of
Holzmann, Kharaghani, and Tayfeh-Rezaie~\cite{Holzmann2008}
and Bright, Kotsireas, and Ganesh~\cite{BKG20-perfect-sequences}
all sets of Williamson sequences (i.e.,
\emph{symmetric} QT sequences)
are known for $n\leq60$.
Up to Williamson-type equivalence, it was previously
determined that there are exactly
94 sets of Williamson sequences for $n\leq21$ and our search
confirms this amount. Moreover, it determines
there are an additional 30 sets of QT sequences for lengths $n\leq21$
that are \emph{not} equivalent to Williamson sequences because they
are only equivalent to sets of sequences that are non-symmetric.

Interestingly, Fitzpatrick
and O'Keeffe~\cite{Fitzpatrick2023} prove that for certain odd
lengths $n$ (when $n$ is prime or of the form
$n=pq$ for distinct odd primes $p$ and $q$ with each
a generator of the other) that \emph{all} sets of Williamson-type
sequences of length $n$ must be Williamson-type equivalent to a set of symmetric sequences.
For $n\leq21$, the only odd lengths not covered by their theorem
are $n=9$ and $n=21=3\cdot7$ (because 7 is not a generator modulo~3).
Our search reveals that their theorem also holds for $n=21$, because
we find no non-symmetric QT sequences of length 21.
For $n=9$, a set of non-symmetric QT sequences exists;
an example was already known to Fitzpatrick
and O'Keeffe, and our enumeration shows up to Williamson-type equivalence this is the \emph{only}
set of non-symmetric QT sequences of length 9.
This is currently the only known set of Williamson-type sequences
of odd length that is not equivalent to a set of symmetric sequences,
though we find 29 examples of sets of
Williamson-type sequences of even length that are not equivalent to a
set of symmetric sequences.

Fitzpatrick and O'Keeffe's result combined with the
exhaustive search for Williamson sequences of length $n=23$
by Baumert and Hall~\cite{Baumert1965}
implies that there is exactly one set of QT sequences of length $23$
up to Williamson-type equivalence.
We found this set generates two inequivalent sets of QT sequences
up to QT equivalence and Hadamard equivalence.

\section{Construction of quaternionic Hadamard matrices}
\label{sec:qhm}

Perfect sequences correspond in a straightforward manner to Hadamard matrices. As explained in \Cref{sec-qhm-context}, such matrices are of interest in communication and information theory. This further motivates a deeper study of quaternionic Hadamard matrices (QHMs). The new sequences discovered through the computational search in \Cref{sec:search} yield new quaternionic Hadamard matrices. We provide examples of such matrices of interest in quantum communication in \Cref{subsection:patterns}. In contrast to the real and complex cases, these matrices may be used to generate uncountably many (equivalent) QHMs; see \Cref{sec-inf-families}. In \Cref{sec-characterization}, we propose an approach for characterizing QHMs and illustrate it in the case of order five matrices. In \Cref{sec-new-qhm}, we formally establish the nonequivalence of matrices resulting from the computational search for perfect sequences.

\subsection{Hadamard matrices from perfect sequences} 
\label{subsection:patterns}

Perfect quaternion sequences can be used to generate quaternionic Hadamard matrices in a straightforward manner. Given such a sequence~$S$ of length~$n$, we may construct an~$n \times n$ circulant quaternionic Hadamard matrix~$M$ by taking the first row of the matrix to be the sequence~$S$, and the rest of the rows to be successive cyclic shifts of the first row. Formally, we define~$M_{ij} \eqdef S_{j-i}$ for~$0\leq i,j<n$, where we take the indices to be modulo~$n$.

As explained in \Cref{sec:results}, our computational search produces numerous
examples of perfect quaternion sequences that are of interest in quantum information and communication.
Consider the quaternionic Hadamard matrices derived from these sequences as described above, and the equivalent normalized matrices. (Recall that normalization is the process of transforming their first row and column to all~$1$s via equivalence operations.)
We obtain normalized quaternionic Hadamard matrices with non-commuting entries
for all lengths $n$ with~$3 < n \le 21$ in the search described in \Cref{sec:results}.

Below we list a few examples of the matrices we find.
The entries are described using symbols given by the mapping
\begin{equation}
\label{eq-qt-pqs-mapping}
\Big[\begin{array}{ccccccccccccc}
\texttt + & \texttt - & \texttt i & \texttt j & \texttt k & \texttt q & \texttt x & \texttt y & \texttt z & \texttt s & \texttt u & \texttt v & \texttt w \\
1 & -1 & \qi & \qj & \qk & q & q\qi & q\qj & q\qk & q^* & q^*\qi & q^*\qj & q^*\qk
\end{array}\Big] , 
\end{equation}
with a capital letter standing for the negation of the quaternion
represented by the same letter in lower case.
For example, \texttt{X} denotes $-q\qi$.
Normalized quaternionic Hadamard matrices for
orders 6, 8, 10, and 12 are given by
\setcounter{MaxMatrixCols}{20}
\begin{gather*}
\begin{bmatrix}
\texttt + &\texttt + &\texttt + &\texttt + &\texttt + &\texttt + \\
\texttt + &\texttt s &\texttt Q &\texttt k &\texttt S &\texttt x \\
\texttt + &\texttt z &\texttt S &\texttt + &\texttt Q &\texttt U \\
\texttt + &\texttt - &\texttt + &\texttt - &\texttt + &\texttt - \\
\texttt + &\texttt S &\texttt Q &\texttt K &\texttt S &\texttt X \\
\texttt + &\texttt Z &\texttt S &\texttt - &\texttt Q &\texttt u \\
\end{bmatrix},
\qquad
\begin{bmatrix}
\texttt + &\texttt + &\texttt + &\texttt + &\texttt + &\texttt + &\texttt + &\texttt +\\
\texttt + &\texttt Y &\texttt I &\texttt q &\texttt - &\texttt y &\texttt i &\texttt Q\\
\texttt + &\texttt I &\texttt - &\texttt i &\texttt + &\texttt I &\texttt - &\texttt i\\
\texttt + &\texttt q &\texttt i &\texttt Y &\texttt - &\texttt Q &\texttt I &\texttt y\\
\texttt + &\texttt - &\texttt + &\texttt - &\texttt + &\texttt - &\texttt + &\texttt -\\
\texttt + &\texttt y &\texttt I &\texttt Q &\texttt - &\texttt Y &\texttt i &\texttt q\\
\texttt + &\texttt i &\texttt - &\texttt I &\texttt + &\texttt i &\texttt - &\texttt I\\
\texttt + &\texttt Q &\texttt i &\texttt y &\texttt - &\texttt q &\texttt I &\texttt Y\\
\end{bmatrix},
\\
\begin{bmatrix}
\texttt+ &\texttt+ &\texttt+ &\texttt+ &\texttt+ &\texttt+ &\texttt+ &\texttt+ &\texttt+ &\texttt+ \\
\texttt+ &\texttt J &\texttt+ &\texttt Y &\texttt U &\texttt- &\texttt j &\texttt- &\texttt y &\texttt u \\
\texttt+ &\texttt Q &\texttt q &\texttt W &\texttt y &\texttt+ &\texttt Q &\texttt q &\texttt W &\texttt y \\
\texttt + &\texttt w &\texttt J &\texttt Z &\texttt + &\texttt - &\texttt W &\texttt j &\texttt z &\texttt - \\
\texttt + &\texttt j &\texttt U &\texttt z &\texttt J &\texttt + &\texttt j &\texttt U &\texttt z &\texttt J \\
\texttt + &\texttt - &\texttt - &\texttt j &\texttt - &\texttt - &\texttt + &\texttt + &\texttt J &\texttt + \\
\texttt + &\texttt J &\texttt - &\texttt y &\texttt v&\texttt +&\texttt J&\texttt - &\texttt y &\texttt  v \\
\texttt + &\texttt Y &\texttt y &\texttt V &\texttt q &\texttt - &\texttt y &\texttt Y &\texttt v &\texttt Q \\
\texttt + &\texttt S &\texttt J &\texttt Z &\texttt - &\texttt + &\texttt S &\texttt J &\texttt Z &\texttt - \\
\texttt + &\texttt j &\texttt v &\texttt z &\texttt J &\texttt - &\texttt J &\texttt V &\texttt Z &\texttt j 
\end{bmatrix},
\mspace{12.5mu}
\begin{bmatrix}
\texttt + &\texttt + &\texttt + &\texttt + &\texttt + &\texttt + &\texttt + &\texttt + &\texttt + &\texttt + &\texttt + &\texttt +\\
\texttt + &\texttt W &\texttt Y &\texttt - &\texttt W &\texttt y &\texttt - &\texttt w &\texttt y &\texttt + &\texttt w &\texttt Y\\
\texttt + &\texttt Y &\texttt W &\texttt - &\texttt y &\texttt w &\texttt + &\texttt Y &\texttt W &\texttt - &\texttt y &\texttt w\\
\texttt + &\texttt - &\texttt - &\texttt + &\texttt + &\texttt - &\texttt - &\texttt + &\texttt + &\texttt - &\texttt - &\texttt +\\
\texttt + &\texttt W &\texttt y &\texttt + &\texttt W &\texttt y &\texttt + &\texttt W &\texttt y &\texttt + &\texttt W &\texttt y\\
\texttt + &\texttt y &\texttt w &\texttt - &\texttt y &\texttt W &\texttt - &\texttt Y &\texttt W &\texttt + &\texttt Y &\texttt w\\
\texttt + &\texttt - &\texttt + &\texttt - &\texttt + &\texttt - &\texttt + &\texttt - &\texttt + &\texttt - &\texttt + &\texttt -\\
\texttt + &\texttt w &\texttt Y &\texttt + &\texttt W &\texttt Y &\texttt - &\texttt W &\texttt y &\texttt - &\texttt w &\texttt y\\
\texttt + &\texttt y &\texttt W &\texttt + &\texttt y &\texttt W &\texttt + &\texttt y &\texttt W &\texttt + &\texttt y &\texttt W\\
\texttt + &\texttt + &\texttt - &\texttt - &\texttt + &\texttt + &\texttt - &\texttt - &\texttt + &\texttt + &\texttt - &\texttt -\\
\texttt + &\texttt w &\texttt y &\texttt - &\texttt W &\texttt Y &\texttt + &\texttt w &\texttt y &\texttt - &\texttt W &\texttt Y\\
\texttt + &\texttt Y &\texttt w &\texttt + &\texttt y &\texttt w &\texttt - &\texttt y &\texttt W &\texttt - &\texttt Y &\texttt W
\end{bmatrix}.
\end{gather*}

These matrices may be used to generate infinite families of Hadamard matrices via algebra automorphisms of the quaternions, as explained in more detail in \Cref{sec-inf-families}.
They may also be used to characterize matrices with the same pattern of entries as in \Cref{sec-characterization}.
However, the matrices so generated may not include all the quaternionic Hadamard matrices of a given order.
We examine this possibility later in the section.

\subsection{Infinite families of fixed order}
\label{sec-inf-families}

For a fixed real or complex Hadamard matrix~$M$, there are only a finite number of distinct normalized matrices equivalent to~$M$. 
\begin{lemma}
\label{lem-finite-equiv-class}
Suppose~$M$ is a real (or complex) Hadamard matrix. The number of normalized real (respectively, complex) Hadamard matrices which are equivalent to~$M$ is finite.
\end{lemma}
\begin{proof}
By \Cref{lem-equivalence} and the note following it, any matrix~$M'$ equivalent to~$M$ may be written as~$D_1 P_1 M P_2 D_2$, where~$P_1$ and~$P_2$ are permutation matrices, and~$D_1$ and~$D_2$ are diagonal matrices of unit real or complex numbers. The lemma is immediate for real matrices, as the number of matrices~$D_1$, $D_2$, $P_1$, and~$P_2$ as above is finite. We prove the lemma for complex matrices. 

Let~$M'' \eqdef P_1 M P_2$. There are a finite number of such matrices~$M''$ obtained by permuting row and columns of~$M$. Suppose~$M'' = (x_{ij} : 0 \le i,j < n)$. Then normalizing~$M''$ gives us the matrix~$N'' \eqdef ( x_{i0}^* x_{ij} x_{0j}^* : 0 \le i,j < n)$.

Suppose~$D_1 = \diag(a)$ and~$D_2 = \diag(b)$. Then~$M' = ( a_i x_{ij} b_j : 0 \le i,j < n)$, and normalizing~$M'$ gives us the matrix~$N' \eqdef ( b_0^* x_{i0}^* x_{ij} x_{0j}^* a_0^* : 0 \le i,j < n )$. Since the first row of~$N'$ is all~$1$, we have~$ b_0^* x_{00}^* a_0^* = 1$, i.e., $b_0^* = a_0 x_{00}$. So~$N' = ( a_0 x_{00} x_{i0}^* x_{ij} x_{0j}^* a_0^* : 0 \le i,j < n )$. As complex multiplication is commutative, the entries of~$N'$ simplify:~$ N' = ( x_{00} x_{i0}^* x_{ij} x_{0j}^* : 0 \le i,j < n ) = x_{00} N''$. Since the number of distinct matrices~$M''$ (and therefore~$N''$) is finite and~$x_{00}$ is an entry of~$M''$, the number of matrices~$N'$ is also finite.
\end{proof}
We now show that uncountably infinite families of normalized QHM of order~$n$ exist whenever there is a normalized QHM of order~$n$ with a nonreal entry.
We begin with a characterization of automorphisms of the quaternions.
Any such automorphism corresponds to an orthogonal transformation (more specifically, a rotation)
in the real vector space spanned by the basic unit quaternions~$\set{\qi, \qj, \qk}$. 
\begin{theorem}[Theorem~2.4.4, page~17, \cite{Rodman14-quaternion-linear-algebra}]
\label{thm-automorphism}
The group of automorphisms of the quaternions is isomorphic to~$\so(3)$, the Special Orthogonal Group in three dimensions (i.e., rotations of the two-sphere).
\end{theorem}
This implies that for any pair of pure imaginary quaternions~$a,b$ of the same norm, there is an automorphism that maps~$a$ to~$b$. This property helps show the following.
\begin{proposition}
\label{prop-infinite-family}
Suppose~$G$ is a normalized QHM of order~$n$ with at least one non-real entry. There are an uncountably infinite number of normalized QHM of order~$n$ that are equivalent to~$G$.
\end{proposition}
\begin{proof}
Let~$\alpha + w$ be an entry of~$G$ with real part~$\alpha$ and non-zero pure imaginary part~$w$.

Let~$S$ be the set of all normalized quaternionic Hadamard matrices that are equivalent to~$G$. Suppose for the sake of contradiction that~$S$ is countable. Let~$Q$ be the set of quaternions~$x$ such that~$x$ is an entry of some matrix in~$S$. The set~$Q$ is also countable. The set of pure imaginary quaternions with the same norm as~$w$ form a sphere of radius~$\norm{w}$ in~$ \Span\set{\qi, \qj, \qk}$. Thus there is a pure imaginary quaternion~$y$ with~$\norm{y} = \norm{w}$ that is different from the imaginary part of any quaternion in~$Q$. By \Cref{thm-automorphism}, there is an automorphism~$f$ of~$\HH$ that maps~$w$ to~$y$. By \Cref{prop-inner-aut}, there is a pure imaginary unit quaternion~$a$ such that~$f(x) = a^* x a$ for every~$x$. So~$f(G)$ is a QHM that is equivalent to~$G$. By construction, the matrix~$f(G)$ has an entry which is not contained in~$Q$, so~$f(G) \not\in S$. This is a contradiction, hence~$S$ is uncountably infinite.
\end{proof}
Thus the Hadamard matrices listed in \Cref{subsection:patterns} lead to infinite families of (equivalent) quaternionic Hadamard matrices. 
The appendix lists the perfect quaternion sequences corresponding to nonequivalent QT sequences of length up to~$21$.
The normalized QHMs derived from perfect sequences found by our computational search typically contain non-real entries, and hence also lead to infinite equivalence classes of QHMs.
As illustrated by the examples in \Cref{subsection:patterns}, the normalized matrices typically also have noncommuting entries.
Such QHMs lead to MUMs of interest in quantum communication theory~\cite{FKN23-mum-superdense-coding}.

\subsection{Quaternionic Hadamard matrices of order five}
\label{sec-characterization}

Chterental and {\DJ}okovi{\'c}~\cite{CD08-stochastic-matrices} identified all quaternionic Hadamard matrices of order four, characterizing them in terms of two parametrized families of QHMs.
They posed the characterization of quaternionic Hadamard matrices of order five as an open problem.
Higginbotham and Worley~\cite{HW22-qhm} make progress on this problem by studying a specific type of normalized order five QHMs with a circulant core,
where the \emph{core} of a normalized matrix is obtained by deleting its first row and column.

Since characterizing all~$5 \times 5$ quaternionic Hadamard matrices appears to be challenging, a natural approach to take is to study matrices with a fixed pattern of entries.
This approach appears to be feasible, and we illustrate it with an example.
Our analysis of this example will ultimately allow us to prove (see \Cref{sec-new-qhm})
that the infinite family of QHMs produced by the single (up to QT equivalence) quadruple of QT sequences of order five
is \emph{not} equivalent to the QHMs considered by Higginbotham and Worley~\cite{HW22-qhm}.

Consider the~$5 \times 5$ quaternionic Hadamard matrix
\begin{equation}
\label{eq-qhm5}
\begin{bmatrix}
1 &1  &1  &1  &1  \\
1 &-1  &c\qj+s\qi  &c\qj-s\qi  &\qj  \\
1 &c\qj+s\qi  &-1  &\qj  &c\qj-s\qi  \\
1 &c\qj-s\qi  &\qj  &-1  &c\qj+s\qi  \\
1 &\qj  &c\qj-s\qi  &c\qj+s\qi  &-1 
\end{bmatrix}
\end{equation}
with $c \eqdef \cos(\frac{2\pi}3)$ and~$s \eqdef \sin(\frac{2\pi}3)$. This corresponds to the~$5 \times 5$ Hadamard matrix of~$2 \times 2$ unitary operators reported by Tavakoli \textit{et al.\/}~\cite{TFR+21-mubs}, via the~$\reals$-linear extension of the anti-homorphism mapping
\begin{equation}
\label{eq:isometry}
1 \leftrightarrow \id, \quad \qi \leftrightarrow -\complexi \PauliX, \quad \qj \leftrightarrow -\complexi \PauliY, \quad \qk \leftrightarrow \complexi \PauliZ,
\end{equation}
where~$\id, \PauliX,\PauliY,\PauliZ$ are the~$2 \times 2$ Pauli operators
\[
\id \eqdef 
\begin{bmatrix}
    1 & 0 \\
    0 & 1
\end{bmatrix}, \ 
\PauliX \eqdef 
\begin{bmatrix}
    0 & 1 \\
    1 & 0
\end{bmatrix}, \ 
\PauliY \eqdef 
\begin{bmatrix}
    0 & -\complexi \\
    \complexi & 0
\end{bmatrix},
\text{ and } 
\PauliZ \eqdef 
\begin{bmatrix}
    1 & 0 \\
    0 & -1
\end{bmatrix},
\]
and~$\complexi \eqdef \sqrt{-1}$ is the complex square root of~$-1$ (see~\cite[Section~II.D]{FKN23-mum-superdense-coding} for the details).
A similar QHM was presented by Higginbotham and Worley~\cite[Example~3]{HW22-qhm} as an example of an order five QHM with a non-circulant core.
In fact, it follows from \Cref{thm-gabc-characterization} below that the QHM in~\cite[Example~3]{HW22-qhm} is equivalent to~\eqref{eq-qhm5}.

The matrix~\eqref{eq-qhm5} has the same pattern of entries as in the matrix
\begin{align}
\label{eq-form1}
G(a,b,c) & \eqdef
\begin{bmatrix}
1 &1  &1  &1  &1  \\
1 &-1  &a  &b  &c  \\
1 &a  &-1  &c  &b  \\
1 &b  &c  &-1  &a  \\
1 &c  &b  &a  &-1 
\end{bmatrix},
\end{align}
where~$a$, $b$, $c \in \qsphere^3$ are unit quaternions.
We would like to characterize~$a$, $b$, and~$c$ such that~$G(a,b,c)$ a Hadamard matrix.
We begin with a simple property of quaternions.

\begin{lemma} 
\label{lemma_sum}
Let $a$, $b\in \qsphere^3$ be unit quaternions such that
    \begin{equation}
        1+a+b \eq 0 . \label{eq:sum}
    \end{equation}
Then there is a quaternionic square root~$w$ of~$-1$, i.e., $w \in \qsphere^2$, such that $a = \e^{\frac{2\pi}{3} w}$ and~$b = \e^{\frac{4\pi}3 w}$.
In other words, $a$ and~$b$ are non-trivial quaternionic cube-roots of~$1$ and~$b = a^2$.
\end{lemma}
\begin{proof}
Let us write $a = a_1 + a_2\qi + a_3\qj + a_4\qk$ and $b = b_1 + b_2\qi + b_3\qj + b_4\qk$. 
Eq.~\eqref{eq:sum} is equivalent to
\begin{equation}
\label{eq:sum-eq} 
    \begin{aligned}
            a_1 + b_1 & = -1 && \text{and} \\
            a_i & = -b_i && \text{for $i \in \{2,3,4\}$}.
    \end{aligned} 
\end{equation}
Furthermore, since $a$ and $b$ are unit quaternions, we have $1 = \norm{a} = \norm{b}$, i.e., $1 = a_1^2 + a_2^2 + a_3^2 + a_4^2 = b_1^2 + b_2^2 + b_3^2 + b_4^2$. Since~$a_i^2 = b_i^2$ for~$i > 1$ by~\eqref{eq:sum-eq}, we have~$a_1^2 = b_1^2$.
Thus $a_1 = b_1 = -\frac12$ and $b = a^*$. Define~$w \eqdef \frac{2}{\sqrt{3}} (a_2\qi + a_3\qj + a_4\qk)$. 
We see that $\norm{w}^2 = \frac{4}{3} (a_2^2 + a_3^2 + a_4^2) =\frac{4}{3}(1 - a_1^2) = 1$. So~$w \in \qsphere^2$. Moreover, according to the Euler formula for quaternions, $\e^{\frac{2\pi}{3} w} = \cos(\frac{2\pi}{3}) + w\sin(\frac{2\pi}{3}) = a$. Similarly, $b = \e^{\frac{4\pi}{3}w}$.
\end{proof}
We may verify that the converse also holds.
That is, for any non-trivial quaternionic cube-root~$a$ of~$1$ ($a^3 = 1$ and~$a \neq 1$) we have~$1 + a + a^2 = 0$.
Using this property, we determine precisely when~$G(a,b,c)$ is a Hadamard matrix.
\begin{theorem}
\label{thm-gabc-characterization}
Consider the matrix\/~$G(a,b,c)$ defined in Eq.~\eqref{eq-form1}, where~$a$, $b$, $c$ are unit quaternions.
Suppose\/~$G(a,b,c)$ is a Hadamard matrix. Then up to an automorphism of\/~$\HH$, it equals
\begin{align}
\begin{bmatrix}
1 &1  &1  &1  &1  \\
1 &-1  &\qi  &\qi \e^{\frac{2\pi}{3} \qj}  & \qi \e^{\frac{4\pi}{3} \qj}  \\
1 &\qi  &-1  &\qi \e^{\frac{4\pi}{3} \qj}  &\qi \e^{\frac{2\pi}{3} \qj}  \\
1 &\qi \e^{\frac{2\pi}{3} \qj}  & \qi \e^{\frac{4\pi}{3} \qj}  &-1  &\qi  \\
1 &\qi \e^{\frac{4\pi}{3} \qj} &\qi \e^{\frac{2\pi}{3} \qj}  &\qi  &-1 
\end{bmatrix} ,
\label{eq-form2}
\end{align}
i.e., $G(a,b,c)$ is equivalent to the matrix~\eqref{eq-form2}.
\end{theorem}

\begin{proof}
For ease of notation, we abbreviate~$G(a,b,c)$ by~$M$.
Suppose~$M$ is a Hadamard matrix. Let~$M_i$ denote the~$i$th row of~$M$ for $0\leq i<5$. The condition~$M M^\adjoint = 5 \,\id_5$ is equivalent to the row-orthogonality conditions~$M_i M_j^\adjoint = 5 \diracdelta_{ij}$, for all~$0\leq i,j < 5$.

From~$M_2 M_1^\adjoint = 0$, we get $a + b + c = 0$, i.e.,~$a( 1 + a^* b + a^* c) = 0$. According to \Cref{lemma_sum}, there is a quaternionic square-root of~$-1$, $ w \in \qsphere^2$, such that $b = a\e^{\frac {2\pi}{3} w}$ and $c = a\e^{\frac {4\pi}{3} w}$. The orthogonality of rows~$M_2$ and~$M_3$ along with the relationship between $a$, $b$, and~$c$ gives
\[
\begin{array}{cll}
      & 1 - a^*-a+bc^*+cb^* & = 0\\
 \iff & 1 - 2\Re(a)+  a \e^{\frac{-2 \pi}{3} w} a^* + a \e^{\frac{2 \pi}{3} w} a^* & = 0 \\
 \iff & 1 - 2\Re(a)+  a\bigl(2 \cos(\frac {2\pi} 3) \bigr) a^* & = 0 \\
 \iff & 1 - 2\Re(a) - \norm{a}^2 & = 0 \\
 \iff & \Re(a) & = 0 .
\end{array}
\]
Therefore, $a \in \qsphere^2$.

From $M_2 M_4^\adjoint = 0$, which is the same as $M_3 M_5^\adjoint = 0$, we get
\[
\begin{array}{cll}
      & 1 - b^*-ac^*-b+ca^* & = 0\\
 \iff & 1 - 2 \Re(b) + a\e^{\frac{-4\pi}{3} w} a^* + a\e^{\frac{4 \pi}{3} w} a^* & = 0 \\
 \iff & 1 - 2 \Re(b) - aa^* & = 0 \\
 \iff &  \Re( a \e^{\frac{2 \pi}{3} w}) & = 0 \\
 \iff & \Re\bigl( a \bigl( - \frac{1}{2} + \frac{\sqrt{3}}{2} w \bigr) \bigr)  & = 0 \\
 \iff & \Re( aw ) & = 0 ,  \\
\end{array}
\]
since~$\Re(a) = 0$.

Define~$x \eqdef aw$. The above derivation shows that~$x \in \qsphere^2$, so~$x^* = -x = -aw$. Since~$a$, $w \in \qsphere^2$, we also have~$-aw = x^* = w^* a^* = (-w) (-a) = wa$. That is,~$a$ and~$w$ are pure imaginary unit quaternions that anticommute. 
In other words, $a$, $w$, and~$x$ satisfy the same relations as~$\qi$, $\qj$, and~$\qk$, and there is an automorphism~$f$ of~$\HH$ given by the mapping
\[
a \leftrightarrow \qi, \quad w \leftrightarrow \qj, \quad \text{and} \quad x \leftrightarrow \qk .
\]
Up to the automorphism~$f$, the matrix~$M$ equals~\eqref{eq-form2}.
We may verify that the matrix~\eqref{eq-form2} satisfies all the remaining row orthogonality conditions, and is therefore a Hadamard matrix.
\end{proof}

Thus, all quaternionic Hadamard matrices of the form~$G(a,b,c)$ are generated from the matrix~$G(\qi, \qi \e^{\frac{2\pi}{3} \qj}, \qi \e^{\frac{4\pi}{3} \qj})$ by automorphisms of the quaternions.
By \Cref{cor:automorphism-equiv}, automorphisms of the quaternions preserve QHM equivalence, so all QHMs of the form $G(a,b,c)$ are equivalent.
Since $G(a,b,c)$ is equivalent to the QHM in~\cite[Example~3]{HW22-qhm} (by swapping
the second row with the second-last row and the middle row with the last row),
the family of QHMs of order five with non-circulant core considered by
Higginbotham and Worley~\cite{HW22-qhm} are also equivalent to QHMs of the form $G(a,b,c)$.
This kind of analysis may be a good starting point for a complete characterization of order five QHMs.

\subsection{Nonequivalent QHMs of small order}
\label{sec-new-qhm}

The QHMs constructed from perfect sequences can lead to matrices not equivalent to those obtained from other constructions.
We now prove that the QHMs constructed from the length five perfect sequence \verb|x+JJ+|
are not equivalent to matrices of the form~\eqref{eq-form1} or to the other order
five QHMs considered by Higginbotham and Worley~\cite{HW22-qhm}.

Consider the length five perfect sequence~\verb|x+JJ+| listed in the appendix and the QHM obtained from this sequence. For ease of comparison, we work with the equivalent circulant matrix
\begin{equation}
\label{eq-m1}
M \eqdef
\begin{bmatrix}
1 & q \qi & 1 & -\qj & -\qj \\
-\qj & 1 & q \qi & 1 & -\qj \\
-\qj & -\qj & 1 & q \qi & 1 \\
1 & -\qj & -\qj & 1 & q \qi \\
q \qi & 1 & -\qj & -\qj & 1 
\end{bmatrix} .
\end{equation}
having~\verb|+x+JJ| as the first row.
We begin by showing that $M$ is not equivalent to the order five matrices of the form~$G(a,b,c)$ studied in \Cref{sec-characterization}. In what follows, we assume the indices are taken to be modulo~$n$, where~$n$ is the length of the perfect sequence or order of the corresponding QHM\@.

\begin{lemma}
\label{lem-eq-cyclic-qhm}
Let\/~$H$ be a cyclic QHM of order\/~$n$ with first row\/~$h$. Let\/~$\tH$ denote the QHM obtained by dephasing the first column of\/~$H$, i.e., $\tH \eqdef \diag(c)^\adjoint H$, where\/~$c$ is the first column of\/~$H$. Any normalized QHM equivalent to\/~$H$ may be expressed as\/~$\alpha h_k Q_1 \tH \diag(g(k))^\adjoint Q_2 \alpha^*$, for some unit quaternion\/~$\alpha$, $k \in [0,n)$, and permutation matrices\/~$Q_1$, $Q_2$, where\/~$g(k)$ is the vector given by the cyclic permutation of the entries of\/~$h$ by\/~$k$ places to the left. 
\end{lemma}
\begin{proof}
Let~$h$ denote the first row of~$H$, so that~$g(k)_i = h_{i+k}$, $H = ( h_{j-i} : 0 \le i,j < n)$ and~$\tH = ( h_{-i}^* h_{j-i} : 0 \le i,j < n)$. 

Suppose~$H'$ is a normalized QHM equivalent to~$H$. By \Cref{lem-equivalence}, we have~$H' = D_1 P_1 H P_2 D_2$ for some diagonal matrices~$D_1$, $D_2$ with unit quaternions on the diagonals and some permutation matrices~$P_1$, $P_2$. We may write~$H' = P_1 D_1' H D_2' P_2$, where~$D_1' \eqdef P_1^\adjoint D_1 P_1$ and~$D_2' \eqdef P_2 D_2 P_2^\adjoint$. Note that~$D_1'$, $D_2'$ are also diagonal matrices with unit quaternions on the diagonals. 

Let~$K \eqdef D_1' H D_2'$. It suffices to prove that~$K$ is of the form stated in the lemma, i.e., $K = \alpha h_k Q_1' \tH \diag(g(k))^\adjoint Q_2' \alpha^* $ for some permutation matrices~$Q_1'$, $Q_2'$ and~$\alpha$, $k$, $g(k)$ as in the lemma.

Let~$D_1' = \diag(a)$, and~$D_2' = \diag(b)$ so that~$K = (a_i h_{j-i} b_j : 0 \le i,j < n)$. Suppose~$P_1$ maps the~$(r+1)$th row of~$K$ to the first row of~$H'$, and~$P_2$ maps the~$(s+1)$th column of~$K$ to the first column of~$H'$. Since~$H'$ is normalized, all entries of the~$(r+1)$th row of~$H'$ and its~$(s+1)$th column are~$1$. So~$a_r h_{j-r} b_j = 1 = a_i h_{s-i} b_s$ for all~$0 \le i,j < n$. Solving these equations, we get~$b_j = h_{j-r}^* a_r^*$ and~$a_i = b_s^* h_{s-i}^* = a_r h_{s-r} h_{s-i}^* \mspace{0.7mu}$. So
\[
K = (a_r h_{s-r} h_{s-i}^* h_{j-i} h_{j-r}^* a_r^* : 0 \le i,j < n) = a_r h_{s-r} K' a_r^* \enspace,
\]
where~$K' = (h_{s-i}^* h_{j-i} h_{j-r}^* : 0 \le i,j < n)$. Note that the~$(i,j)$th entry of~$K'$ is the~$(i-s, j-s)$th entry of~$\tH \diag(g(k))^\adjoint$, with~$k \eqdef s-r$. The lemma follows.
\end{proof}

The lemma helps us prove that the matrix~$M$ is not equivalent to the QHM~$G(\qi, \qi \e^{\frac{2\pi}{3} \qj}, \qi \e^{\frac{4\pi}{3} \qj})$, and therefore to any QHM of the form~$G(a,b,c)$.
\begin{proposition}
\label{prop-noneq-gabc}
The QHM\/~$M$ in Eq.~\eqref{eq-m1} is not equivalent to the QHM\/ $G(\qi, \qi \e^{\frac{2\pi}{3} \qj}, \qi \e^{\frac{4\pi}{3} \qj})$ as defined in \Cref{sec-characterization}.
\end{proposition}
\begin{proof}
Observe that every column of~$G(\qi, \qi \e^{\frac{2\pi}{3} \qj}, \qi \e^{\frac{4\pi}{3} \qj}) $ other than the first contains exactly three pure imaginary quaternions. We argue that for any normalized QHM equivalent to~$M$, there is some column different from the first that does not have this property.

By \Cref{lem-eq-cyclic-qhm}, any normalized matrix equivalent to~$M$ has the form
\[
\alpha h_k Q_1 \tH \diag(g(k))^\adjoint Q_2 \alpha^*
\]
for~$\alpha$, $h$, $k$, $Q_1$, $\tH$, $g(k)$, and~$Q_2$ as in the lemma. Every inner automorphism~$x \mapsto \alpha x \alpha^*$ preserves pure imaginary quaternions, and the permutation of rows of a matrix preserves the number of pure imaginary quaternions in each of the columns. The permutation~$Q_2$ maps some column of the matrix~$L \eqdef h_k \tH \diag(g(k))^\adjoint $ to the first and permutes the rest. So it suffices to prove that there is some other column of~$L$ that does not contain exactly three pure imaginary quaternions.

Consider the third column of~$L$.
Depending on the value of~$k$, this column is given in the table
\begin{equation}
\label{eq-3rd-column}
\begin{array}{c|cccccccccc}
k & ~~ & 0 & ~~ & 1 & ~~ & 2 & ~~ & 3 & ~~ & 4 \\ \hline
h_k h_0^* h_2 h_{2 + k}^* && 1           && q \qk     && \qj       && -\qj        && -\qk q^* \\
h_k h_4^* h_1 h_{2 + k}^* && \qj q \qi   && (q \qk)^2 && \qj q \qk && q \qi       && 1 \\
h_k h_3^* h_0 h_{2 + k}^* && \qj         && - q \qi   && -1        && 1           && -\qi q^* \\
h_k h_2^* h_4 h_{2 + k}^* && -\qj        && q \qi     && 1         && -1          && \qi q^* \\
h_k h_1^* h_3 h_{2 + k}^* && \qi q^* \qj && 1         && -\qi q^*  && \qk q^* \qj && (\qk q^*)^2
\end{array}.
\end{equation}
Since none of the possible columns is all~$1$, the permutation~$Q_2$ does not map the third column of~$L$ to the first. We may verify that none of the above column vectors have exactly three pure imaginary quaternions. The claim follows.
\end{proof}

Recall that the order~$n$ complex Fourier matrix~$F_n$ is given by~$F_n(i,j) \eqdef \upomega^{ij}$, $0 \le i,j < n$, where~$\upomega \eqdef \exp(2\pi \qi/n)$. The matrix~$F_n$ is a quaternionic Hadamard matrix for all~$n \ge 1$. The complex Fourier matrix~$F_5$ has a circulant core and is contained in a class of order five QHMs with a circulant core characterized by Higginbotham and Worley~\cite[Theorem~6]{HW22-qhm}. We now prove the matrix~$M$ is not equivalent to this class of QHMs.
\begin{proposition}
\label{prop-noneq-circulant}
The QHM\/~$M$ in Eq.~\eqref{eq-m1} is not equivalent to QHMs with a circulant core and with two complex entries in the core.
\end{proposition}
\begin{proof}
By~\cite[Theorem~6]{HW22-qhm}, any order five QHM with a circulant core and with two complex entries in the core is either equivalent to the complex Fourier matrix~$F_5 \mspace{0.7mu}$, or has two pairs of non-real conjugate quaternions~$(\beta, \beta^*)$, $(\gamma, \gamma^*)$ in every column other than the first. Moreover, in the latter case, the inner product of the vectors of coefficients of the pure imaginary parts of~$\beta$ and~$\gamma$ is~$\pm 5/16$.

Every inner automorphism~$x \mapsto \alpha x \alpha^*$ preserves the conjugacy relation. By \Cref{thm-automorphism}, it also preserves the inner products of the vectors of coefficients of the pure imaginary parts of every pair of quaternions. Similarly, two quaternions~$\theta$, $\omega$ commute if and only if~$\alpha \theta \alpha^*$ and~$\alpha \omega \alpha^*$ commute.
Also note that two quaternions commute if and only if their pure imaginary parts are real multiples of the same quaternion.

Consider the nonequivalence of~$M$ to the complex Fourier matrix~$F_5 \mspace{0.7mu}$. Note that all the entries of the complex Fourier matrix commute. By following the same arguments as in the proof of \Cref{prop-noneq-gabc}, it suffices to show that the pure imaginary parts of the third column of the QHM~$L$ are not all real multiples of the same quaternion. By inspecting the possibilities for this column in Eq.~\eqref{eq-3rd-column}, we see that this is indeed the case.

Consider the nonequivalence of~$M$ to the other class of order five QHM described above. Inspecting the possibilities for the third column of~$L$ in Eq.~\eqref{eq-3rd-column}, we see that only the possibility with~$k = 0$ has two pairs of non-real conjugate quaternions. However, the inner product of the vectors of coefficients of the pure imaginary parts of the non-conjugate quaternions is~$\pm 1/2$. The claim follows.
\end{proof}

Next, we show there are at least three nonequivalent QHMs of order seven. 
Consider the order seven QHM~$R$ obtained from the sequence~\verb|+JYJ+--| reported in the appendix (up to a shift), 
\begin{equation}
\label{eq-o7qhm}
R \eqdef
\begin{bmatrix}
 1 & -\qj & -q \qj & -\qj & 1 & -1 & -1 \\
-1 & 1 & -\qj & - q \qj & -\qj & 1 & -1 \\
-1 & -1 & 1 & -\qj & - q \qj & -\qj & 1 \\
 1 & -1 & -1 & 1 & -\qj & - q \qj & -\qj \\
-\qj & 1 & -1 & -1 & 1 & -\qj & - q \qj \\
- q \qj & -\qj & 1 & -1 & -1 & 1 & -\qj \\
-\qj & - q \qj & -\qj & 1 & -1 & -1 & 1 
\end{bmatrix} .
\end{equation}
We compare $R$ to the QHMs corresponding to $\Q_{24}$-sequences reported in prior works, viz.~\verb|+jIIj+W|~\cite[Example~7.3]{Kuznetsov2010} and~\verb|qJIKKIJ|~\cite{LS11-williamson-matrices} (see~\cite[Table~1]{BarreraAcevedo2024}).
The former sequence cannot be directly converted into a set of QT sequences because of the presence of the
\verb|W| entry, since $q^*\qk\notin\Q_+$.  The application of conjugation converts \verb|+jIIj+W| into a perfect $\Q_+$-sequence
whose corresponding set of QT sequences is QT equivalent to the set of QT sequences corresponding to \verb|qJIKKIJ|.
However, conjugation is not an automorphism of~$\HH$ (it is an \emph{anti}-automorphism, i.e., $(xy)^*=y^*x^*$), so conjugation does not in
general preserve the property of being a QHM and is not an equivalence operation of QHMs.
Thus, below we prove that $R$ is not equivalent
to both the normalized form of the QHM obtained from~\verb|+jIIj+W|,
\begin{equation}
\label{eq-o7qhm-k10}
S \eqdef
\begin{bmatrix}
1 & 1 & 1 & 1 & 1 & 1 & 1 \\
1 & -\qk q \qj & -\qk q \qk & \qk q & \qk q \qk & \qk q \qj & (\qk q)^2 \\
1 & -q^* \qi & \qi & -\qk & \qk & -\qi & \qi q \\
1 & -1 & \qj q^* \qj & \qk & -\qj & -\qk & q \\
1 & \qi & -1 & -\qi q^* \qj & -\qk & \qk & \qk q \\
1 & -\qj & \qj & -1 & -\qi q^* \qi & \qi & -q \\ 
1 & -\qi & -\qj & \qi & -1 & \qj q^* \qk & -\qi q 
\end{bmatrix} ,
\end{equation}
and the normalized form of the QHM obtained from~\verb|jQjikki| (a shift and negation of \verb|qJIKKIJ|),
\begin{equation}
\label{eq-o7qhm-badl}
T \eqdef
\begin{bmatrix}
1 & 1 & 1 & 1 & 1 & 1 & 1 \\
1 & \qi q^* & \qk q \qj & -1 & -\qi & \qk & \qi \\
1 & -q^* & -\qi & -\qi q \qi & -1 & -\qk & \qk \\
1 & -\qj q^* & -\qj & \qj & -\qi q \qk & -1 & -\qi \\
1 & q^* & \qj & \qk & -\qj & \qk q \qk & -1 \\
1 & -\qi q^* & \qi & -\qj & \qj & -\qi & q \qi \\
1 & (\qj q^*)^2 & \qj q^* \qk & \qj q^* \qj & -\qj q^* & -\qj q^* \qj & -\qj q^* \qk
\end{bmatrix} .
\end{equation}
In order to prove nonequivalence of~$R$ to~$S$ and~$T$, we follow the approach taken for order five matrices.
\begin{proposition}
\label{prop-o7-nonequiv}
The QHM\/~$R$ in Eq.~\eqref{eq-o7qhm} is neither equivalent to the QHM\/~$S$ defined in Eq.~\eqref{eq-o7qhm-k10} nor to the QHM\/~$T$ defined in Eq.~\eqref{eq-o7qhm-badl}.
\end{proposition}
\begin{proof}
By \Cref{lem-eq-cyclic-qhm}, any normalized matrix equivalent to~$R$ has the form
\[
\alpha h_k Q_1 \tH \diag(g(k))^\adjoint Q_2 \alpha^*
\]
for~$\alpha$, $h$, $k$, $Q_1$, $\tH$, $g(k)$, and~$Q_2$ as in the lemma. Let~$L \eqdef h_k \tH \diag(g(k))^\adjoint $.
It suffices to prove that there is some column of~$L$ that does not correspond to any column of~$S$ or~$T$ (after the application of the permutations~$Q_1, Q_2$ and automorphism~$x \mapsto \alpha x \alpha^*$ as in the above expression).

Consider the second column of~$L$.
Depending on the value of~$k$, this column equals one of the columns of the table
\begin{equation}
\label{eq-2nd-column}
\begin{array}{c|cccccccccccccc}
k & 0 && 1 && 2 && 3 && 4 && 5 && 6 \\
\hline
h_k h_{0}^* h_{1} h_{1+k}^*  & 1 && -\qj q^* && -q \qj && -1 && \qj && -\qj && \qj \\
h_k h_{6}^* h_{0} h_{1+k}^*  & -\qj && -q^* && -q && \qj && 1 && -1 && 1 \\
h_k h_{5}^* h_{6} h_{1+k}^*  & \qj && q^* && q && -\qj && -1 && 1 && -1 \\
h_k h_{4}^* h_{5} h_{1+k}^*  & -\qj && -q^* && -q && \qj && 1 && -1 && 1 \\
h_k h_{3}^* h_{4} h_{1+k}^*  & -1 && \qj q^* && q \qj && 1 && -\qj && \qj && -\qj \\
h_k h_{2}^* h_{3} h_{1+k}^*  & \qj q^* && (q^*)^2 && 1 && -q^* \qj && \qj q^* \qj && -\qj q^* \qj && \qj q^* \qj \\
h_k h_{1}^* h_{2} h_{1+k}^*  & \qj q  && 1 && q^2 && -q \qj && \qj q \qj && -\qj q \qj && \qj q \qj 
\end{array},
\end{equation}
since the~$i$th entry of the column for~$k \in [0,7)$ is~$h_k h_{-i}^* h_{1-i} h_{1+k}^*$.
The column is not all~$1$, and has one of the following types depending on the value of~$k$. 
\begin{enumerate}
\item \label{item-p1}
It has two elements in~$\set{\pm 1}$ and three elements in~$\set{ \pm \beta}$, where~$\beta \in \Q_8$ is some pure imaginary quaternion.

\item \label{item-p2}
It has three elements in~$\set{\pm 1}$ and two elements in~$\set{ \pm \beta}$, where~$\beta \in \Q_8$ is some pure imaginary quaternion.

\item \label{item-p3}
It has two distinct pairs of quaternions of the form~$(\beta, -\beta)$ for some~$\beta \in \Q_{24} \setminus \Q_8$, and one pair of quaternions~$(\gamma, \gamma^2)$ where~$\gamma \in \Q_{24} \setminus \Q_8$ is a cube root of~$-1$. 

\end{enumerate}

All elements of~$R$, $S$, $T$ are in~$\Q_{24}$, but the rest of the proof only uses the property that the elements of~$ \Q_{24} \setminus \Q_8$ have non-zero real and imaginary parts. Suppose~$R$ is equivalent to~$S$ or~$T$. Then the automorphism~$f \colon x \mapsto \alpha x \alpha^* $ (and the permutations~$Q_1$, $Q_2$) preserves the properties of the three types of columns described above. However, none of the columns of~$S$ or~$T$ are of types~\ref{item-p1} or~\ref{item-p2}. Only the last column of~$S$ and the second column of~$T$ are of type~\ref{item-p3}. So the second column of~$L$ may only be mapped to these columns, and with~$k \in \set{1,2}$ (cf. Eq.~\eqref{eq-2nd-column}). We argue next that this too is not possible.

We present the argument for the last column of~$S$ and~$k = 1$. The other cases follow along the same lines.
The only pairs of the form~$(\beta, -\beta)$ in the last column of~$S$ are~$(q,-q)$ and~$(\qi q, -\qi q)$, and in the second column of~$L$ with~$k = 1$ are~$(q^*, -q^*)$ and~$(\qj q^*, -\qj q^*)$. Since~$\Re(q) = \Re(-\qi q) = \Re(q^*) = \Re(\qj q^*) = 1/2$, the equivalence operations may only map~$(q^* \mapsto q, ~ \qj q^* \mapsto -\qi q)$ or~$(q^* \mapsto -\qi q, ~ \qj q^* \mapsto q)$. Further, we have~$-q^* \mapsto \qk q$ or~$-q^* \mapsto (\qk q)^2$. However, the two sets of conditions are not consistent as they map~$q^*$ to distinct quaternions, so the required equivalence operations do not exist.
\end{proof}

Using similar arguments as above, we show that the QHMs~$R$ and~$S$ are not equivalent to the complex Fourier matrix~$F_7 \mspace{0.7mu}$. 
\begin{proposition}
\label{prop-fourier-noneq}
Neither the QHM\/~$R$ in Eq.~\eqref{eq-o7qhm} nor the QHM\/~$S$ defined in Eq.~\eqref{eq-o7qhm-k10} is equivalent to the order seven Fourier matrix~$F_7$.
\end{proposition}
\begin{proof}
As in the proof of \Cref{prop-o7-nonequiv}, to show that~$R$ is not equivalent to~$F_7\mspace{0.7mu}$, it suffices to show that none of the columns in~\eqref{eq-2nd-column} equal a column of~$F_7$ after some permutation of rows and the application of an automorphism. All the entries of~$F_7$ commute. However, every column in~\eqref{eq-2nd-column} has at least two entries whose pure imaginary parts are not real multiples of the same quaternion, and this property is preserved by the permutation of rows and the application of an automorphism. The claim for~$R$ follows.

We may repeat this argument for the nonequivalence of~$S$ to~$F_7 \mspace{0.7mu}$. Let~$L' \eqdef h_k \tH \diag(g(k))^\adjoint $, where the parameters~$h, \tH, g(k)$ now correspond to the QHM generated by the perfect sequence~\verb|+jIIj+W|, and~$k$ is determined by a putative sequence of equivalence operations. Consider the second column of~$L'$. Depending on the value of~$k$, this column equals one of the columns of the table
\begin{equation}
\label{eq-2nd-column-s}
\begin{array}{c|cccccccccccccc}
k & 0 && 1 && 2 && 3 && 4 && 5 && 6 \\
\hline
h_k h_{0}^* h_{1} h_{1+k}^*  & 1 && -\qi && -\qj  && -\qi  && -1  && -\qi q  && -q^* \qi  \\
h_k h_{6}^* h_{0} h_{1+k}^*  & \qk q \qj && -\qi q \qi  && -\qj q \qi  && \qj q \qj  && -\qi q  && (\qk q)^2  && 1  \\
h_k h_{5}^* h_{6} h_{1+k}^*  & q^* \qi && \qj q^* \qj  && -\qi q^* \qj  && -\qi q^* \qi  && \qj q^* \qk  && 1  &&  (q^* \qk)^2 \\
h_k h_{4}^* h_{5} h_{1+k}^*  & -1 && \qi && \qj  && \qi  && 1  && \qi q  && q^* \qi \\
h_k h_{3}^* h_{4} h_{1+k}^*  & \qi && -1  && -\qk  && 1  && \qi  && q  && -q^*  \\
h_k h_{2}^* h_{3} h_{1+k}^*  & -\qj && -\qk  && 1  && \qk  && \qj  && -\qk q  && q^* \qk  \\
h_k h_{1}^* h_{2} h_{1+k}^*  & -\qi && 1  && \qk  && -1  && -\qi  && -q  && q^* 
\end{array} .
\end{equation}
Again, every column in~\eqref{eq-2nd-column-s} has at least two entries whose pure imaginary parts are not real multiples of the same quaternion. The claim for~$S$ also follows.
\end{proof}

The results in this section suggest that the multiplicity of equivalence classes of QHMs extends to larger orders as well and we conjecture that this is the case for all orders~$\ge 4$.

\section{Conclusion}

We have enumerated perfect quaternion sequences over $\Q_+$ of lengths $n\leq21$.
In the process, we derived new properties of quaternion-type Hadamard matrices with circulant blocks,
formalized three different types of equivalence of such objects, and developed methods ensuring a complete enumeration
up to each notion of equivalence.  In particular, the natural notion of equivalence of quadruples of Williamson-type sequences does not preserve the natural
notion of equivalence of Hadamard matrices, but introducing a third notion of equivalence allowed us to classify all
quaternion-type Hadamard matrices with circulant blocks up to equivalence of Hadamard matrices.
Ultimately, we classify all quaternion-type Hadamard matrices with circulant blocks of orders $n\leq21$ up to each
type of equivalence we consider (see \Cref{table:results}).

We also proved and exploited a one-to-one correspondence between quadruples of circulant Williamson-type matrices and quaternion-type Hadamard matrices with circulant blocks.
Both these kinds of Hadamard matrices have long been studied~\cite{Golomb1963,Wallis1973}.
It is immediate that
circulant Williamson-type matrices produce Hadamard matrices of quaternion-type, but surprisingly
it was not previously known that quaternion-type Hadamard matrices with circulant blocks
necessarily have blocks that are of Williamson-type.
We prove this in \Cref{thm:qt-wt-equiv}.  This property
allows our search for quaternion-type Hadamard matrices defined by QT sequences
to filter the vast majority of pairs of sequences from consideration early.

Our computational enumeration of QT sequences helps construct new quaternion-type Hadamard matrices with circulant blocks and provably nonequivalent quaternionic Hadamard matrices.
We also study properties of quaternionic Hadamard matrices analytically, and demonstrate the feasibility of characterizing quaternionic Hadamard matrices with a fixed pattern of entries.
Perfect sequences give us quaternionic Hadamard matrices that are provably not equivalent to those constructed by other means, as we show in \Cref{sec-new-qhm}.
These results indicate a larger abundance of quaternionic Hadamard matrices exist than known about in previous works,
including matrices studied for their application to quantum communication protocols~\cite{FKN23-mum-superdense-coding}.

\bibliographystyle{plainurl}
\bibliography{references}

\section*{Appendix: List of QT sequences}
\label{sec-qt-sequences}

Here we provide an exhaustive list of all QT sequences up to Williamson-type equivalence
for all lengths $n\leq21$.  The sequences are classified as either symmetric (such sequences
are also known as \emph{Williamson sequences}) or non-symmetric.
Each set of QT sequences is encoded as a
perfect quaternion sequence using the correspondence from \Cref{thm-bad-correspondence}, as described in \Cref{sec:sequences}.
The symbols \texttt{+} and \texttt{-} denote $1$ and $-1$, respectively; \texttt{q}, \texttt{i}, \texttt{j}, \texttt{k} denote the quaternions~$q=\frac{1+\qi+\qj+\qk}{2}$, $\qi$, $\qj$, $\qk$, respectively;
\texttt{x}, \texttt{y}, \texttt{z} denote $q\qi$, $q\qj$, $q\qk$, respectively;
and capitalization corresponds to the negation of the quaternion denoted by the same letter in lower case. For example, \texttt{X} denotes $-q\qi$.

\microtypesetup{protrusion=false}
\subsubsection*{Length 1 (1 total; 1 symmetric, 0 non-symmetric)}
\raggedright
\verb|+|
\subsubsection*{Length 2 (1 total; 1 symmetric, 0 non-symmetric)}
\raggedright
\verb|+J|
\subsubsection*{Length 3 (1 total; 1 symmetric, 0 non-symmetric)}
\raggedright
\verb|Q++|
\subsubsection*{Length 4 (2 total; 1 symmetric, 1 non-symmetric)}
\raggedright
\emph{1 Symmetric:}\\
\verb|++-+|
\\[0.5\baselineskip]
\emph{1 Nonsymmetric:}\\
\verb|+YIQ|\\
\subsubsection*{Length 5 (1 total; 1 symmetric, 0 non-symmetric)}
\raggedright
\verb|x+JJ+|
\subsubsection*{Length 6 (1 total; 1 symmetric, 0 non-symmetric)}
\raggedright
\verb|KJ+j+J|
\subsubsection*{Length 7 (2 total; 2 symmetric, 0 non-symmetric)}
\raggedright
\verb|YJ+--+J|, \verb|y+JKKJ+|
\subsubsection*{Length 8 (3 total; 1 symmetric, 2 non-symmetric)}
\raggedright
\emph{1 Symmetric:}\\
\verb|J++-J-++|
\\[0.5\baselineskip]
\emph{2 Nonsymmetric:}\\
\verb|+YYIiqq-|, \verb|+JYZikqx|\\
\subsubsection*{Length 9 (4 total; 3 symmetric, 1 non-symmetric)}
\raggedright
\emph{3 Symmetric:}\\
\verb|yi+JKKJ+i|, \verb|XKJ+jj+JK|, \verb|Z+J+--+J+|
\\[0.5\baselineskip]
\emph{1 Nonsymmetric:}\\
\verb|+YQJYZikx|\\
\subsubsection*{Length 10 (2 total; 2 symmetric, 0 non-symmetric)}
\raggedright
\verb|KJ+ikIki+J|, \verb|K+J+j-j+J+|
\subsubsection*{Length 11 (1 total; 1 symmetric, 0 non-symmetric)}
\raggedright
\verb|z-+JKIIKJ+-|
\subsubsection*{Length 12 (5 total; 3 symmetric, 2 non-symmetric)}
\raggedright
\emph{3 Symmetric:}\\
\verb|Y++y-+y+-y++|, \verb|YQ++-QyQ-++Q|, \verb|JIIY+-j-+YII|
\\[0.5\baselineskip]
\emph{2 Nonsymmetric:}\\
\verb|+++Yy++-+yy-|, \verb|+++YQ++-+yQ-|\\
\subsubsection*{Length 13 (4 total; 4 symmetric, 0 non-symmetric)}
\raggedright
\verb|x-+JK+II+KJ+-|, \verb|XKJ+jK--Kj+JK|, \verb|QJ++jJjjJj++J|, \verb|X-k+JKIIKJ+k-|
\subsubsection*{Length 14 (5 total; 5 symmetric, 0 non-symmetric)}
\raggedright
\verb|KJ+J+j-J-j+J+J|, \verb|QQ+J+jzYzj+J+Q|, \verb|+QY+kzxJxzk+YQ|, \verb|YKJ+j-iQi-j+JK|, \verb|QJ++iIkxkIi++J|
\subsubsection*{Length 15 (4 total; 4 symmetric, 0 non-symmetric)}
\raggedright
\verb|XIij+JKJJKJ+jiI|, \verb|QjJKj+JJJJ+jKJj|, \verb|XIKJ+kKjjKk+JKI|, \verb|Q+J+jj+--+jj+J+|
\subsubsection*{Length 16 (18 total; 6 symmetric, 12 non-symmetric)}
\raggedright
\emph{6 Symmetric:}\\
\verb|-J+++-+j-j+-+++J|, \verb|QJ++q-+jQj+-q++J|, \verb|KJ++k-+jKj+-k++J|, \verb|JJ++j-+jJj+-j++J|, \verb|YJ++y-+jYj+-y++J|, \verb|+J++--+j+j+--++J|
\\[0.5\baselineskip]
\emph{12 Nonsymmetric:}\\
\verb|+Y+JYZ+qIxJkqQjy|, \verb|++JJ+J+j+-jJ+j-j|, \verb|++JY+J+y+-jY+j-y|, \verb|++J++J+-+-j++j--|, \verb|++JK+J+k+-jK+j-k|, \verb|+++J+J-++-+j+j--|, \verb|++YJ+Jy++-Yj+jy-|, \verb|++JK+Kj++-Jk+kj-|, \verb|++JJ+Jj++-Jj+jj-|, \verb|+Q++JQYqKXJjixxz|, \verb|++YJIJQ++-YjIjQ-|, \verb|++YIIIQ++-YiIiQ-|\\
\subsubsection*{Length 17 (4 total; 4 symmetric, 0 non-symmetric)}
\raggedright
\verb|ZK+J+i-kKKk-i+J+K|, \verb|yjKJ-+J+KK+J+-JKj|, \verb|xkKIk+JJKKJJ+kIKk|, \verb|y-+Kk+JJKKJJ+kK+-|
\subsubsection*{Length 18 (28 total; 23 symmetric, 5 non-symmetric)}
\raggedright
\emph{23 Symmetric:}\\
\verb|ZXY+i-KZyYyZK-i+YX|, \verb|Z+J+i-KIjYjIK-i+J+|, \verb|YJ++xyKiIZIiKyx++J|, \verb|YJY+yxJZkXkZJxy+YJ|, \verb|Q+++i-IiIXIiI-i+++|, \verb|I+Y+yqIQiJiQIqy+Y+|, \verb|K+J+jiKIkIkIKij+J+|, \verb|ZKJ++jKk-Y-kKj++JK|, \verb|KIJY+jQi-J-iQj+YJI|, \verb|YQJ++jJiyXyiJj++JQ|, \verb|YJQ++jJxkXkxJj++QJ|, \verb|QKJ++jJIjzjIJj++JK|, \verb|+QY+-JJZyjyZJJ-+YQ|, \verb|KJ++yZijkJkjiZy++J|, \verb|KZJ++kKiyIyiKk++JZ|, \verb|ZKJ++kKi-X-iKk++JK|, \verb|QIYJ+iIXjxjXIi+JYI|, \verb|QY++yXKIzxzIKXy++Y|, \verb|Z+J+xQKIkxkIKQx+J+|, \verb|ZKJ++kKI-q-IKk++JK|, \verb|XXY+-KkzqZqzkK-+YX|, \verb|ZKJ+j+j-iXi-j+j+JK|, \verb|QJ++jikKiyiKkij++J|
\\[0.5\baselineskip]
\emph{5 Nonsymmetric:}\\
\verb|+YXXXQY+kKqQyyXqK-|, \verb|+YJXXIY+zKqIyyJqKy|, \verb|++YYXXQ+KkJqxYyyJ-|, \verb|++YJXIQ+JjJqiYjyJ-|, \verb|++YQJY+jqiIzxkZikx|\\
\subsubsection*{Length 19 (6 total; 6 symmetric, 0 non-symmetric)}
\raggedright
\verb|Qj-J-++JKIIKJ++-J-j|, \verb|YKk-Kj+J+JJ+J+jK-kK|, \verb|Y+-kJj+JKKKKJ+jJk-+|, \verb|ZIKJJ+-kiIIik-+JJKI|, \verb|ZIKJ+kJj-JJ-jJk+JKI|, \verb|ZJJJ+j+-jJJj-+j+JJJ|
\subsubsection*{Length 20 (24 total; 17 symmetric, 7 non-symmetric)}
\raggedright
\emph{17 Symmetric:}\\
\verb|JQ+Y+k+x-XjX-x+k+Y+Q|, \verb|JIY++xKkzJjJzkKx++YI|, \verb|IJY++xKkzIiIzkKx++YJ|, \verb|IIYJ+xkiQKiKQikx+JYI|, \verb|YJXY+yjQz+y+zQjy+YXJ|, \verb|XJ+Y+y-ZkJxJkZ-y+Y+J|, \verb|QJ+Y+jkY-IqI-Ykj+Y+J|, \verb|+KJ++kjKIi-iIKjk++JK|, \verb|JQJ++xkKKzjzKKkx++JQ|, \verb|JKJ++xKjkJjJkjKx++JK|, \verb|KKJ++zjJK-k-KJjz++JK|, \verb|IJXY+q-ZX-i-XZ-q+YXJ|, \verb|QKJ++y-+jKqKj+-y++JK|, \verb|YJ++jxJ+-JyJ-+Jxj++J|, \verb|QJJ++x-+jJqJj+-x++JJ|, \verb|QKJ++z-+jKqKj+-z++JK|, \verb|YJ++jyJ+-JyJ-+Jyj++J|
\\[0.5\baselineskip]
\emph{7 Nonsymmetric:}\\
\verb|+++JKQzJK++-+jKqzjK-|, \verb|++JQQ+Jj+j-+jQq+jj-j|, \verb|++JKJQ+Ky++-JkJq+ky-|, \verb|++JJJQ+Jx++-JjJq+jx-|, \verb|++JKJQ+Kx++-JkJq+kx-|, \verb|+++JKQyJK++-+jKqyjK-|, \verb|++YJ+Jj+Jy-+yJ-JJ+jy|\\
\subsubsection*{Length 21 (7 total; 7 symmetric, 0 non-symmetric)}
\raggedright
\verb|ZiIIk-+JKK++KKJ+-kIIi|, \verb|x+iJ--+JK+II+KJ+--Ji+|, \verb|zK+ikj+JKJIIJKJ+jki+K|, \verb|zK+jki+JKIJJIKJ+ikj+K|, \verb|Z++J+-jjJ-++-Jjj-+J++|, \verb|XK+J+kj+I-ii-I+jk+J+K|, \verb|ZIKJ+j-kKiKKiKk-j+JKI|

\end{document}